\documentclass[a4paper]{article}
\usepackage{amsmath}
\usepackage{amsthm}
\usepackage{amssymb}
\usepackage{amsthm}
\usepackage{graphicx}
\usepackage{tikz}
\usepackage{braids}
\usepackage{authblk}
\usepackage{hyperref}

\newtheorem{theorem}{Theorem}[section]
\newtheorem{lemma}[theorem]{Lemma}
\newtheorem{corollary}[theorem]{Corollary}

\theoremstyle{definition}
\newtheorem{definition}[theorem]{Definition}
\newtheorem{example}[theorem]{Example}

\theoremstyle{remark}
\newtheorem{remark}[theorem]{Remark}

\title{Constructing real rational knots by gluing}

\author[1]{Shane D'Mello}
\author[2]{Rama Mishra}
\affil[1]{\href{mailto:shane.dmello@iiserpune.ac.in}{\small \tt shane.dmello@iiserpune.ac.in}}
\affil[2]{\href{mailto:r.mishra@iiserpune.ac.in}{\small \tt r.mishra@iiserpune.ac.in}}

\begin{document}
\maketitle
\begin{abstract}
  We show that the problem of constructing a real rational knot of a reasonably low degree can be reduced to an algebraic problem involving the pure braid group: expressing an associated element of the pure braid group in terms of the standard generators of the pure braid group. We also predict the existence of a real rational knot in a degree that is expressed in terms of the edge number of its polygonal representation.
\end{abstract}
\section{Introduction}

An explicit knot parametrization for knots in $S^3$ has been useful in estimating some important numerical knot~\cite{kuiper}
invariants such as bridge number, superbridge number and geometric degree. It
was shown that each knot in $S^3$ can be parametrized by an embedding from $\mathbb{R}$ to $\mathbb{R}^3$ of the form $t\to (f(t),g(t),h(t))$ where $f(t),g(t)$ and $h(t)$ are real polynomials~\cite{shastri}. The image of such an embedding is a long knot which upon one point compactification gives our knot in $S^3$. In this connection the question of obtaining the polynomials of minimal degree~\cite{mishra1, mishra2} for a given knot type was explored.
%%% SHANE: Consider, ``The connected sum of two polynomial knots cannot be a
%%% polynomial knot...'''
Polynomial parametrizations have certain disadvantages. For instance, polynomial
parametrizations of two knots $K_1$ and $K_2$ cannot be combined to give a
polynomial parametrization of the connected sum $K_1 \# K_2$.  On the other hand
if  we take the projective closure of a  polynomial knot of degree $d$  in
$\mathbb{R}^3$ we obtain a knot in the real projective 3-space $\mathbb{RP}^3$
that intersects the plane at infinity  at one point with multiplicity $d$ which
upon small perturbation turns out to be projective closure of an embedding of
the form $t\to (f_1(t),g_1(t),h_1(t))$ where $f_1$, $g_1$ and $h_1$ are rational
functions.  So, in general, we can study the projective closure of the
embeddings $t \to (f(t),g(t),h(t))$ where $f,g$ and $h$ are rational functions. These are knots in $\mathbb{RP}^3$ (for general, non-algebraic, links in $\mathbb{RP}^3$ see~\cite{drobotukhina1990analogue}). Therefore we define:

\begin{definition}
 A knot in $\mathbb{RP}^3$ is said to be a \emph{real rational knot} of degree $d$ if its parametrization can be realized by a rational map $k:\mathbb{RP}^1 \to \mathbb{RP}^3$, i.e. there exist homogenoeus polynomials $p_0, p_1, \ldots, p_3$, all of the same degree~$d$, so that for any $[s,t]\in \mathbb{RP}^1$, $k(s,t) = [p_0(s,t), p_1(s,t), \ldots, p_3(s,t)]$. $k$ is called the \emph{real rational representative} of the knot.
\end{definition}

  Using a Weierstrass approximation like argument, it can be shown that all
  knots (smooth) in $\mathbb{RP}^3$ are isotopic to the projective closure of a
  real rational knot, but it gives no information about the degree. Therefore, a
  natural problem is to classify all the possible real rational knots of a given
  degree. This was done for low degrees in~\cite{bjorklund}. To classify, one
  can find restrictions on the number or type of isotopy classes and find
  examples to account for those that are permitted by the known restrictions. In
  this paper we focus constructing examples; more specifically, we focus on the
  ``gluing technique'' introduced by Bj{\"o}rklund in \cite{bjorklund} to construct real rational representatives of each of the (projective) knots with less than 5 crossings. However, unlike the examples in \cite{bjorklund}, we aim for methods involving gluing that produce real rational representatives for either all knots or interesting classes containing infinite knots, with a control on the required degree.

The gluing construction~\cite{bjorklund} allows one to glue two knots,
intersecting transversally at a single point, to form a knot with the sum
of the degrees. In his paper that introduces gluing, Bjorklund constructed
examples of real rational knots of low degrees  by gluing conics and lines.
Three natural questions arise:
\begin{enumerate}
 \item Can one obtain a representative for any knot by gluing only lines and/or conics?
\item Can one predict the number of lines/conics required and
  therefore the required degree, in terms of some number associated to the
  classical knot?
\item Can one find a general method of gluing certain families of knots (like
  the Pretzel knots), and hopefully obtain a lower bound on the required degree. \end{enumerate}
 The upper bound can be in terms of a number associated to a classical knot,
 for instance the minimal edge number of a polygonal representation or the
 minimal number of certain words in a braid group representation. For specific
 classes of knots like the $(k_1,\ldots, k_n)-$~pretzel knot, one may get a bound in terms of the $k_i$. 

There are two apparent difficulties when trying to glue conics or lines to
obtain all possible knots. We will demonstrate that both these difficulties can be overcome owing to some basic properties that the gluing construction possesses.

The first difficulty is that conics and lines are geometrically rigid, making
local changes difficult in real rational knots, where as in classical knots, it
is the topology and not the geometry, which matters. For instance, while gluing,
one has to account for the possibility of additional crossings being imposed on
us in a particular diagram. Nevertheless, these crossings can be rendered
harmless by crossing changes~(see Theorem~\ref{switchingcrossings}), which is
also obtainable by gluing. The price one pays for that is an increase in degree
by 2, for each crossing change. Therefore, it becomes important to ensure that
either additional crossings are avoided, or that they can be removed by
Reidemeister moves. The latter is equivalent to finding a suitable
representative of the isotopy class, which is not necessarily the ``standard''
or simplest one. We will see that the braid group approach achieves the latter, i.e. it finds a suitable isotopic representation of a knot which can be constructed by gluing and does not require
any crossing changes.

The second difficulty is that the two knots must be glued at precisely one
point. Gluing at more than one point would inevitably require a ``self gluing''.
Gluing takes two distinct knots that intersect transversally and smoothens the
intersection algebraically. One may ask if we can algebraically smoothen the
self-intersection of a single singular knot, thereby making the method more flexible. This is not possible because if a real rational knot intersects itself, any self-gluing (smoothening of the self-intersection) will result in a curve of genus bigger than 0 and therefore not rational.

Despite these difficulties, we will see that the gluing construction is flexible
enough and can be connected with certain perspectives in classical knot theory. We begin with some simple but important theorems, the first of which is that it is possible to switch a crossing by gluing an ellipse which will increase the
degree of the knot by 2. Secondly, by choosing an appropriate orientation of a circle, one can glue the circle to form a twist. This has been used implicitly by
Bj\"orklund in his constructions of knots with crossings that are less than or
equal to 4 and using this method, it is easy to construct all the pretzel knots and (2,k)-torus knots.

We then turn our attention to a general construction of all possible knots.

We will demonstrate that the isomorphism between the quotient by the braid group
by the pure braid group with the permutation group, along with the standard
generators of the braid group perfectly fit together to allow the gluing of all
knots. We will predict the existence of a real rational representative in a
degree that can be expressed in terms of the minimal number of pure braid group generators.

Polygonal representations of knots are the most geometric and indeed we will
demonstrate two ways of constructing real rational knots using polygonal
representations and obtaining a bound in terms of the edge number, connecting
the complexity of polygonal representations, which is measured in terms of the edge number, with the complexity of real rational representatives, which is measured in terms
of the degree.

Each method has its advantages and disadvantages and it is unlikely that any one
method will consistently produce a knot of a degree lower than the others.

\section{The Basics of Gluing}

\begin{figure}[t]
\begin{center}
  \begin{tabular}{cc}
  \includegraphics[width=4cm]{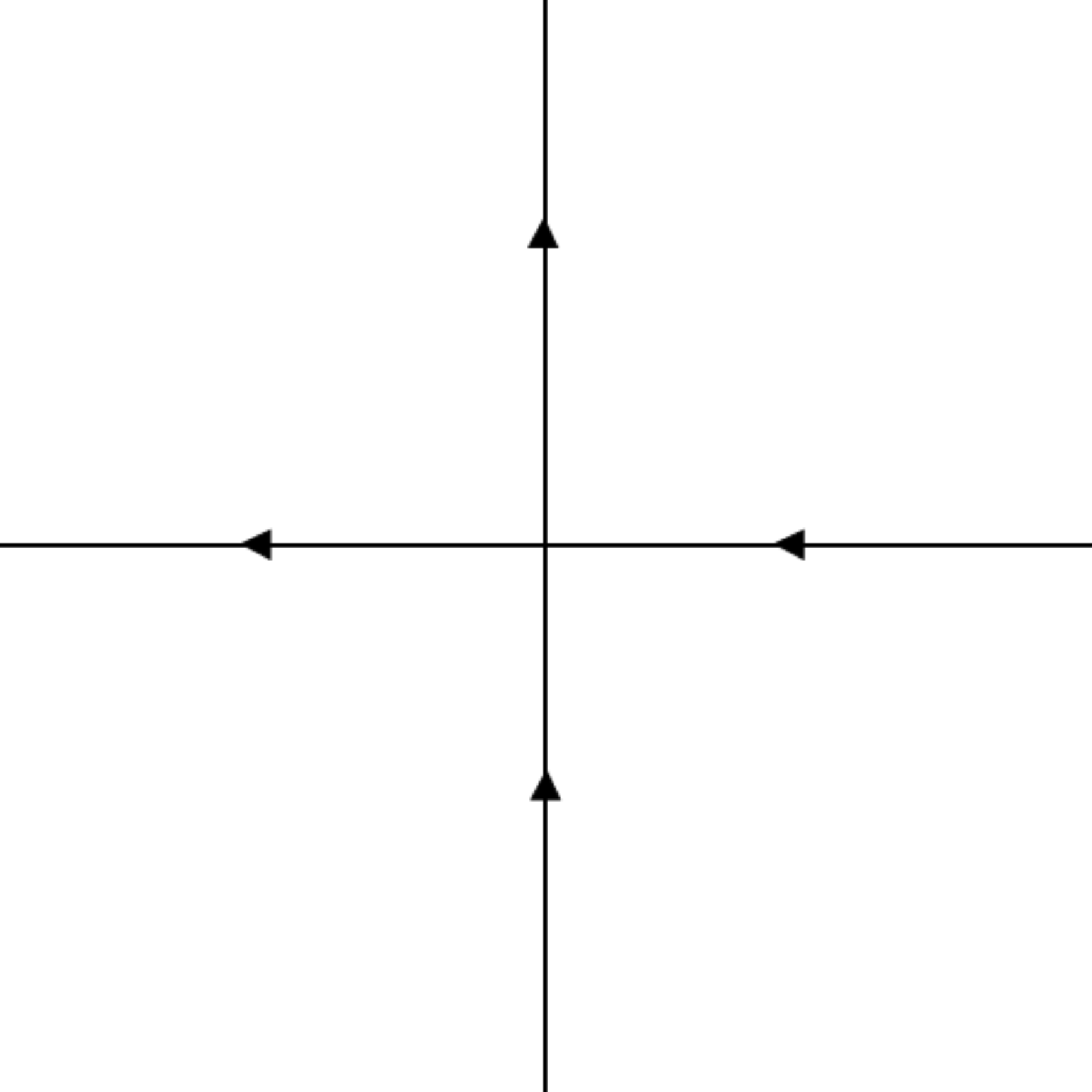} &
  \includegraphics[width=4cm]{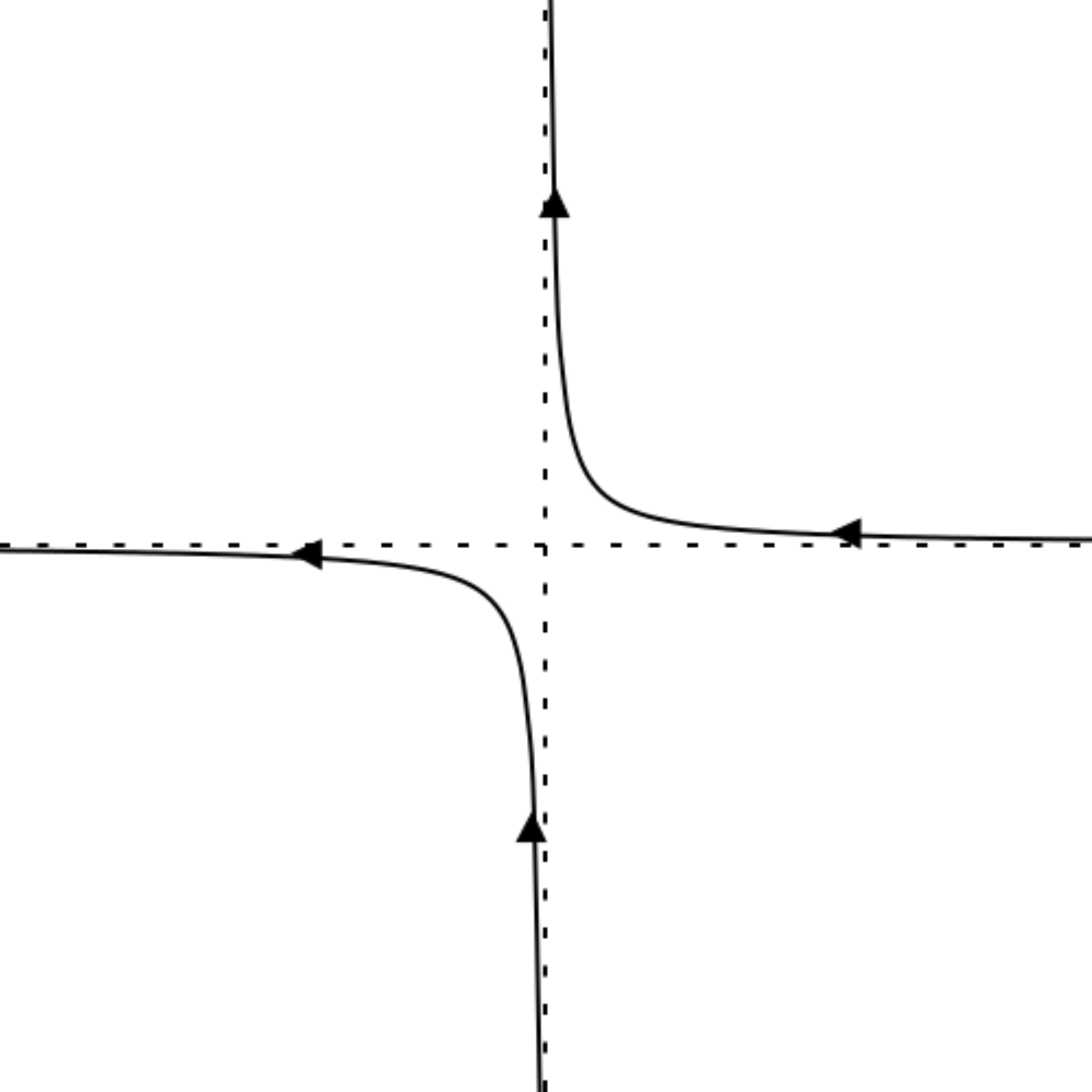} \\
 Two parametrizations& Glued parametrization \\
  \end{tabular}
  \caption{Alternative viewpoint for gluing}
\label{fig:gluingAlternate}
\end{center}
\end{figure}
The following theorem is reworded from~\cite{bjorklund}

\begin{theorem}
  \label{maingluingtheorem}
  Consider two real rational knots $k_1$ of degree~$d_1$ and $k_2$ of
  degree~$d_2$ that intersect in a point $p$. Then there exists a real rational knot $k$ of degree~$d_1+d_2$ called the glued curve, which has the following properties:

\begin{enumerate}
  \item Except for a small neighbourhood $U$ around $p$, the knot $k$ is a section
    of a tubular neighbourhood of the union of $k_1$ and $k_2$. 
  \item One can choose coordinates of $U$ so that $k_1 \cup k_2$ in $U$ is
    like a pair of straight lines intersecting at $p$ while $k$ is a hyperbola.
\item There exist one point on each curve, both of which, remain unchanged in the glued curve. Furthermore, the tangent lines to these points are parallel to the tangent lines of the original curves, and the ratio of their magnitudes can be chosen to be equal.
\end{enumerate}

\end{theorem}

  The proof of 1 and 2 is in~\cite{bjorklund}.  For 3, we will
  reformulate gluing in a way that is more suitable for it and can also be an
  alternative proof for 1 and 2.

\begin{proof}
Let $k_1 : \mathbb{RP}^1 \to \mathbb{RP}^3$  and $k_2 : \mathbb{RP}^1 \to
\mathbb{RP}^3$ be two parametrizations of knots which intersect at a point.
Choose coordinates so that the point of intersection is $[1:0:0:0]$.
Also choose coordinates of the domains, so that $k_i(1,0) = [1:0:0:0]$ for
$i=1,2$.

Consider the projective plane with homogeneous coordinates $[x_0, x_1, x_2]$.
Then the lines defined by $x_1=0$ and $x_2=0$, being copies of the projective
line, may be used as domains for the parametrizations $k_1$ and $k_2$,
respectively (figure~\ref{fig:gluingAlternate}).

The line $x_1=0$ is parametrized by $e_1 (s,t) = [s : t : 0]$ and the line $x_2=0$ is parametrized by $e_2 (s,t) = [s : 0 : t]$. Therefore, given a parametrizations $k_1$ and $k_2$ we can define maps $\tilde{k_i} : \{x_i=0\} \to \mathbb{RP}^3$ by $\tilde{k_i} = k \circ e_i^{-1}$.

Now consider the map $\theta (s,t) \to [st : \epsilon_1 t^2, \epsilon_2 s^2]$, which parametrizes
the hyperbola $x_1x_2=\epsilon_1 \epsilon_2 x_0^2$. Denote this hyperbola by $H$.

Let $\tilde{k_1}= [f_0: f_1: f_2: f_3]$ and $\tilde{k_2}= [g_0:g_1:g_2:g_3]$, then the function $F
= [f_0g_0: f_1g_0 + f_0g_1: f_2g_0 + f_0g_2 : f_0g_3+f_3g_0]$, is an extension
of $\tilde{k_1}$ and $\tilde{k_2}$ to all of $\mathbb{RP}^2$. Observe that $k_i
= F \circ e_i$. 

Let $U$ be a small neighbourhood around $[1:0:0:0]$. Owing to the continuity of
$F$, for a small enough $\epsilon_1$ and $\epsilon_2$, the restriction of $F$ to
the $H \setminus U$ $\theta$ lies in the tubular neighbourhood of the union of
the images of $k_1$ and$k_2$. Therefore, $F \circ \theta$  defines the glued
curve. Indeed, it is straightforward to check that $F \circ \theta$ is precisely
the parametrization of the glued curve as given in~\cite{bjorklund}. 

$H$ is tangential to each of the two lines $\{x_1 = 0\}$ and $\{x_2 =0\}$ at the points $\theta (0,1) = [0, 1, 0]$ and $\theta (1, 0) = [0, 0, 1]$ respectively.
Therefore, $F \circ \theta(1,0) = k_1(1,0)$ and $F \circ \theta(0,1) =
k_2(0,1)$ and the directions of the respective tangents are preserved. By choosing appropriate $\epsilon_1$ we can ensure that the magnitudes are also the same.
\end{proof}

\begin{figure}[t]
\begin{center}
  \begin{tabular}{cc}
  \includegraphics[width=4cm]{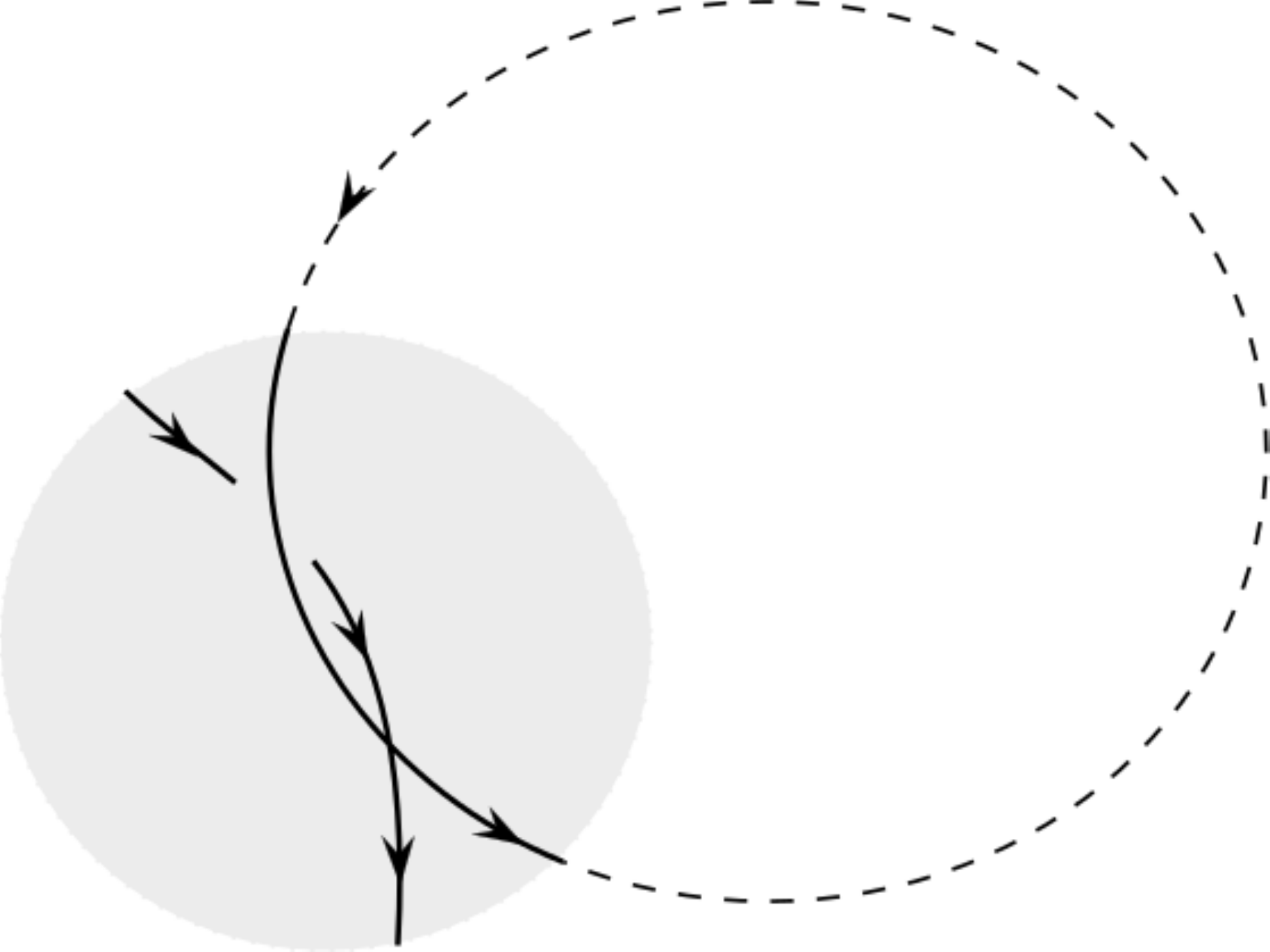} &
  \includegraphics[width=4cm]{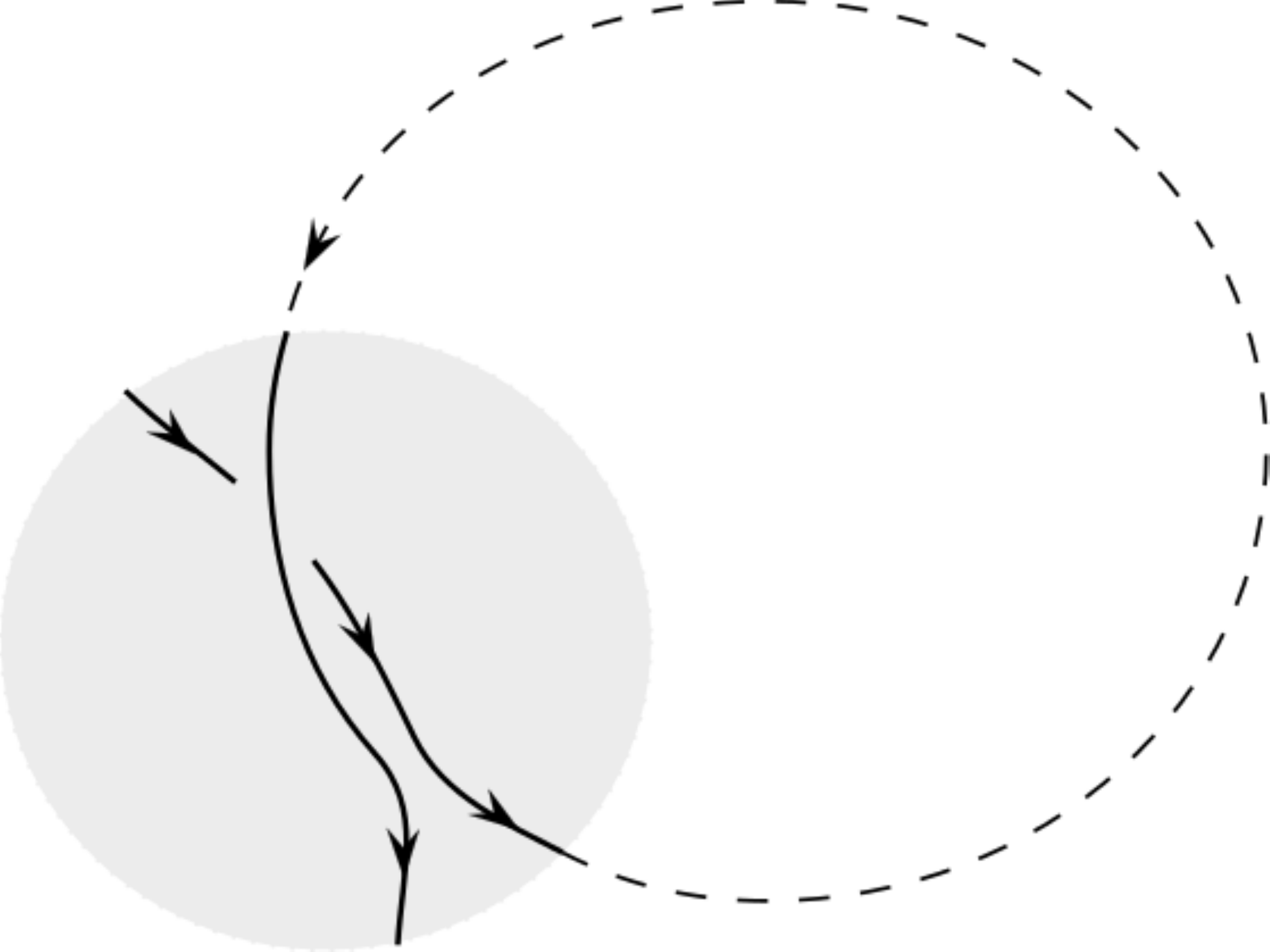} \\
circle oriented clockwise & twist \\
& \\
  \includegraphics[width=4cm]{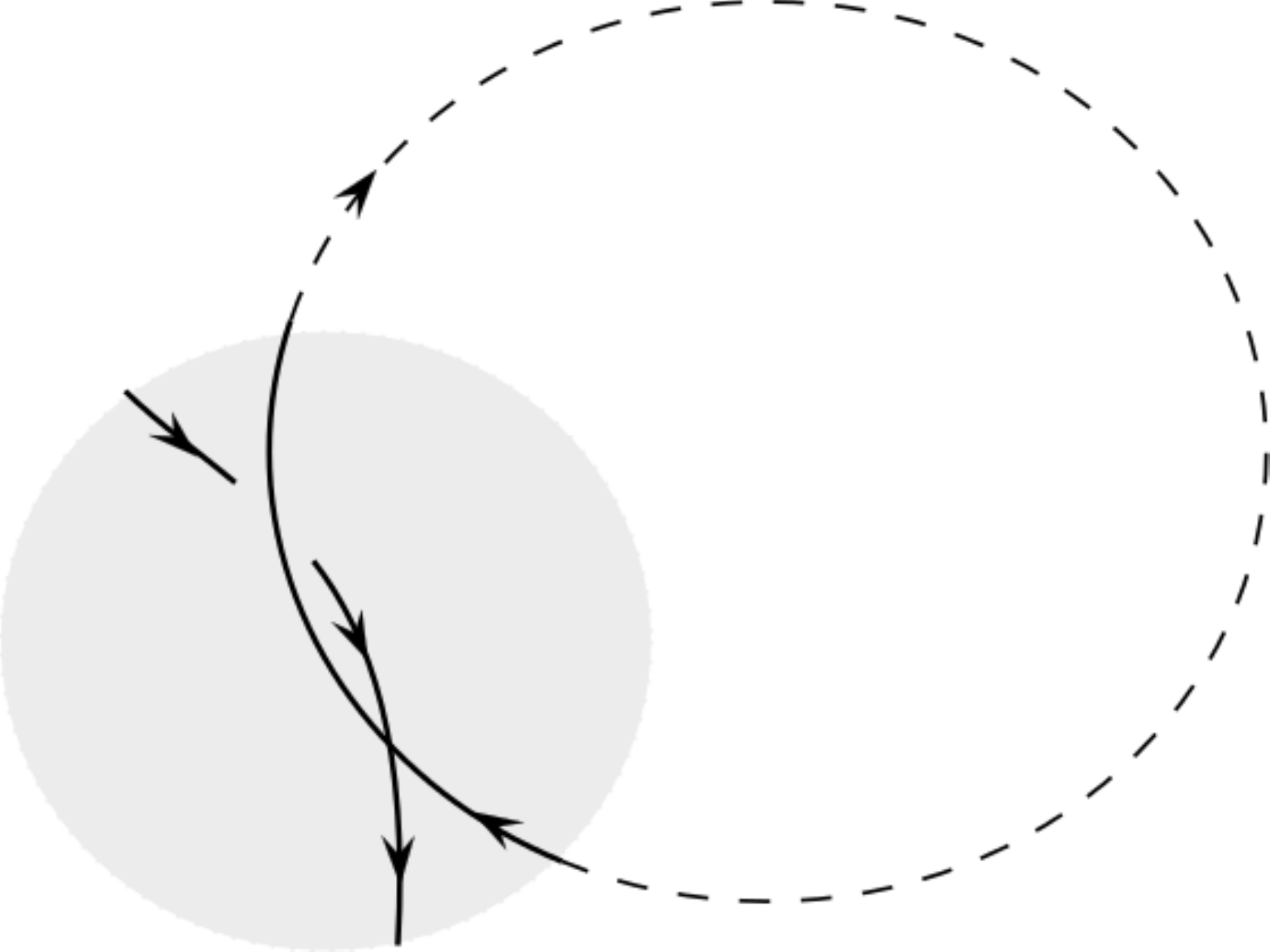} &
  \includegraphics[width=4cm]{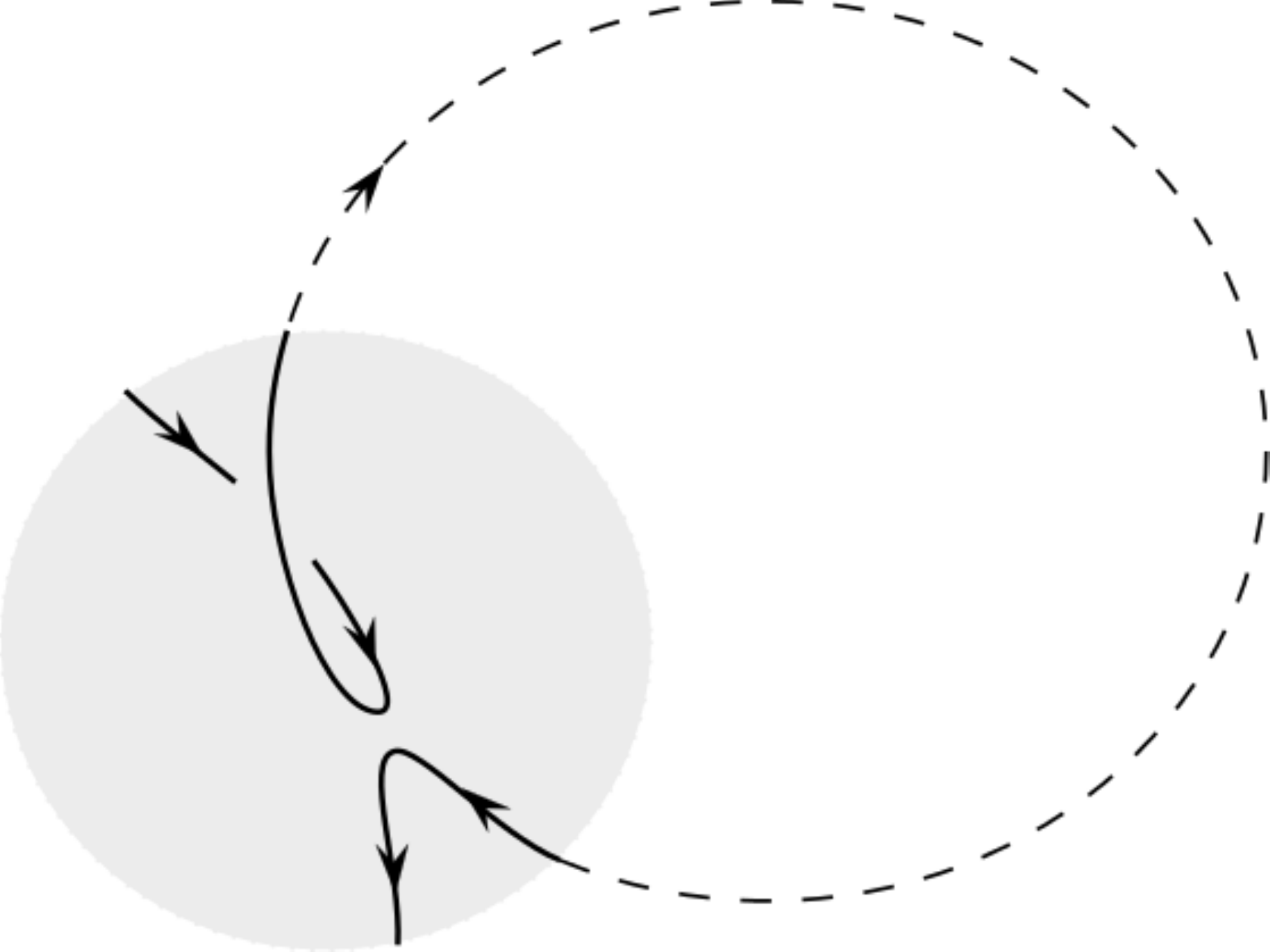}\\
circle oriented ant-clockwise & untwist \\
  \end{tabular}
  \caption{Gluing a circle to form a twist or untwist. Reversing the orientation of the circle changes a twist to an untwist.}
\label{fig:twistuntwist}
\end{center}
\end{figure}

\begin{theorem}
  Let   $k$ be a real rational knot in $\mathbb{RP}^3$, $p \in \mathbb{RP}^3$ a point  that is not on $k$, and $\pi_p$ denote the projection from the point $p$ to $\mathbb{RP}^3$. Then one can glue  a circle $C$ in $\mathbb{RP}^3$ to $k$ at a point $q$ so that $\pi(C \cup k)$, along with the over crossing and under-crossing information, is the digram of a twist. 
\end{theorem}

\begin{proof}
  Glue the circle so that the projection is tangent to the projection of $k$ at the
  point $q$. Then, by a slight perturbation one can ensure that the projection
  of the circle intersects $k$ in $q$ and a point $q'$ that lies in a small enough
  neighbourhood of $q$. Now it is easy to see that on gluing the double point
  $q$ is resolved in one of two ways in the projection, depending on the
  orientation of $C$. Now $q'$ appears either as an over or under-crossing. In
  each case we get one of the cases shown in Figure~\ref{fig:twistuntwist}; the other
  is obtained by reversing the orientation of $C$.
\end{proof}

\begin{figure}[h]
  \begin{center}
 \includegraphics[height=4cm]{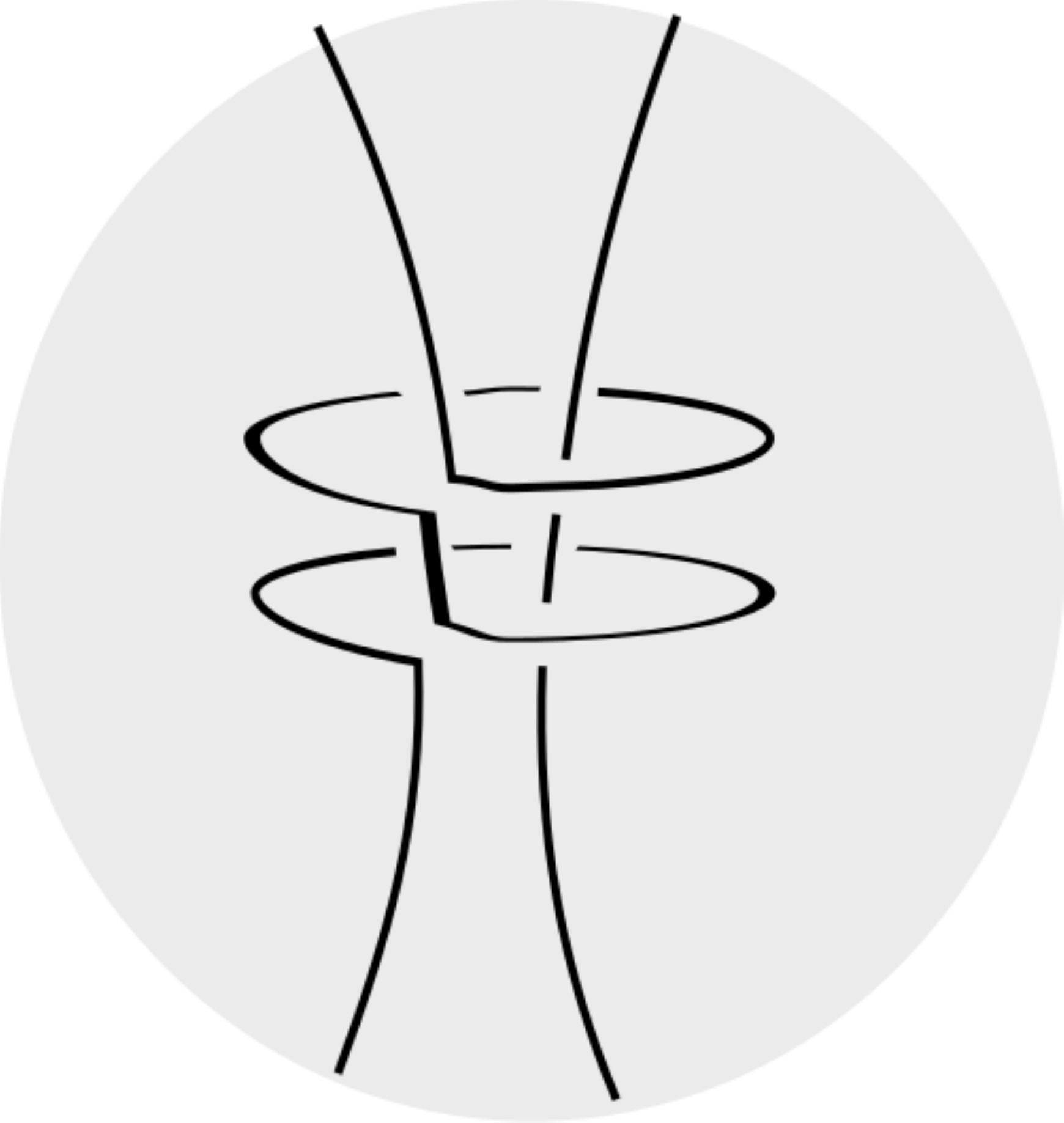} 
 \caption{Two double twists}
 \label{fig:doubletwist}
  \end{center}
\end{figure}

\begin{theorem}
 Consider a real rational knot and suppose that there exists a ball which
 intersects it in two non intersecting lines. By gluing ellipses, one can introduce double twists to the strands.
\end{theorem}
\begin{proof}
Figure~\ref{fig:doubletwist} shows how one can glue ``horizontal'' ellipses inside the ball to introduce double twists.
\end{proof}

\begin{figure}[h]
  \begin{tabular}{ccccc}
  \includegraphics[width=2cm]{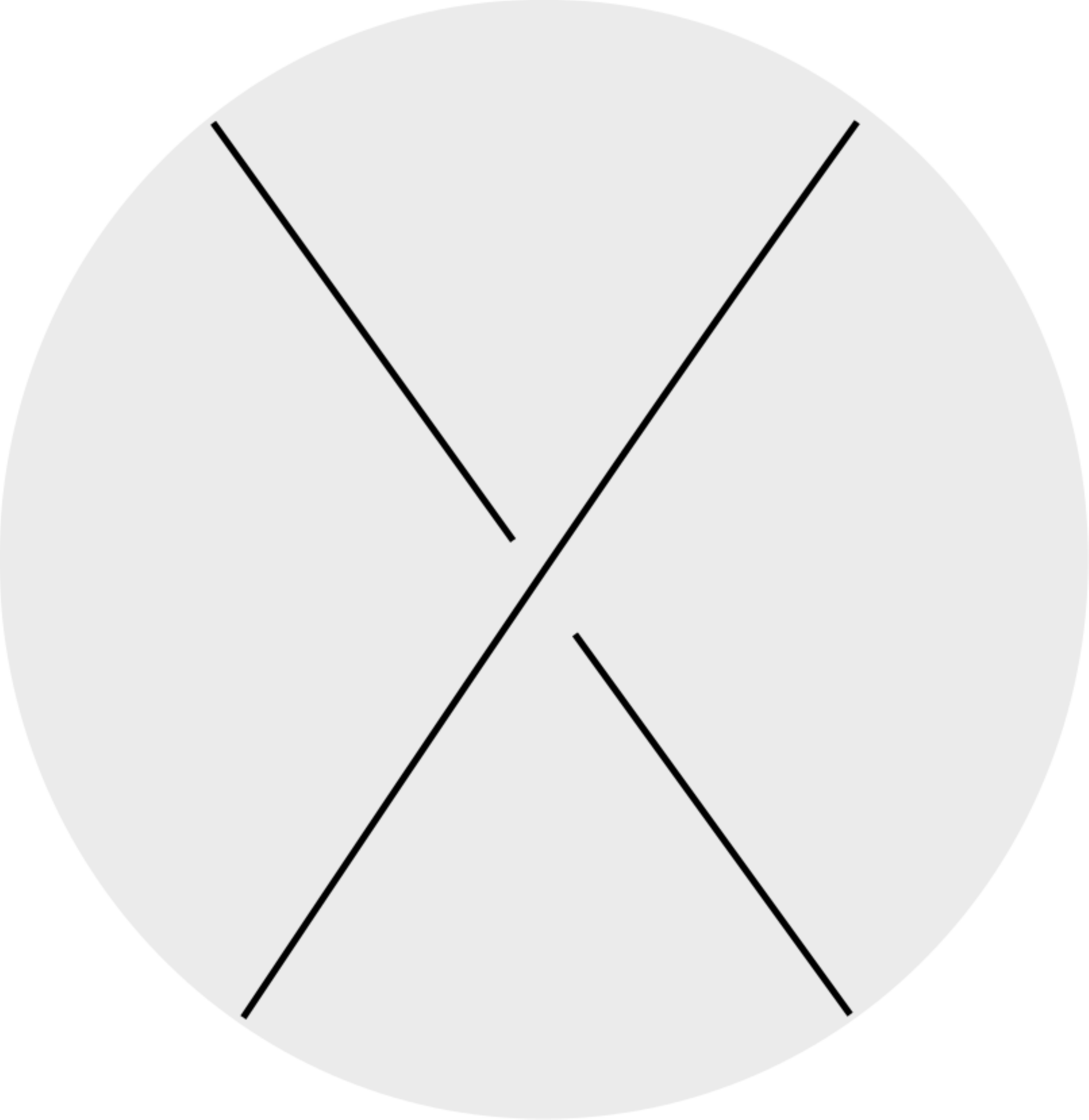}& 
{\LARGE$\xrightarrow{\mathrm{gluing}}$} &
  \includegraphics[width=2cm]{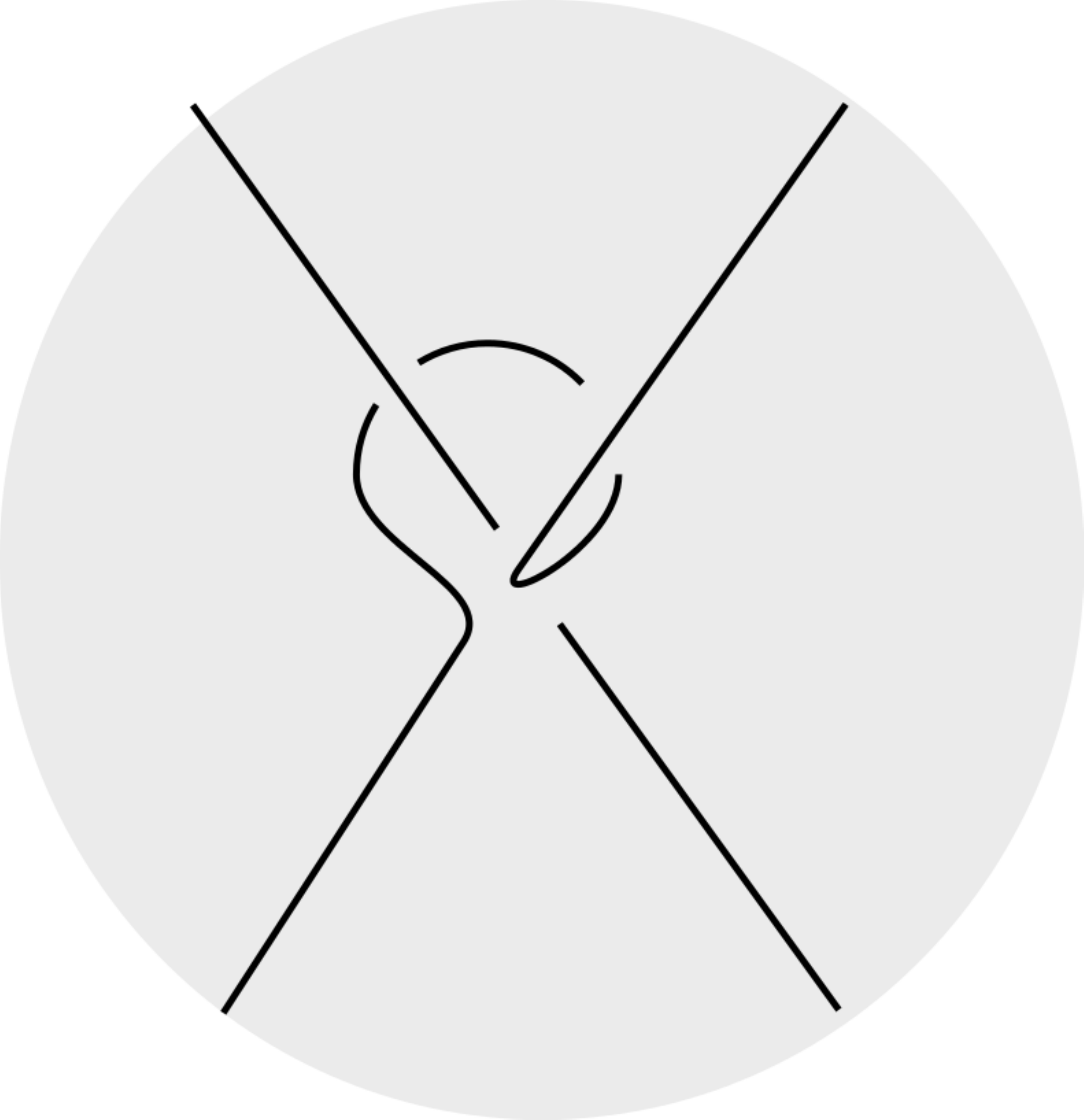}&
{\LARGE$\xrightarrow{\mathrm{1st\ Reidemeister}}$} &
  \includegraphics[width=2cm]{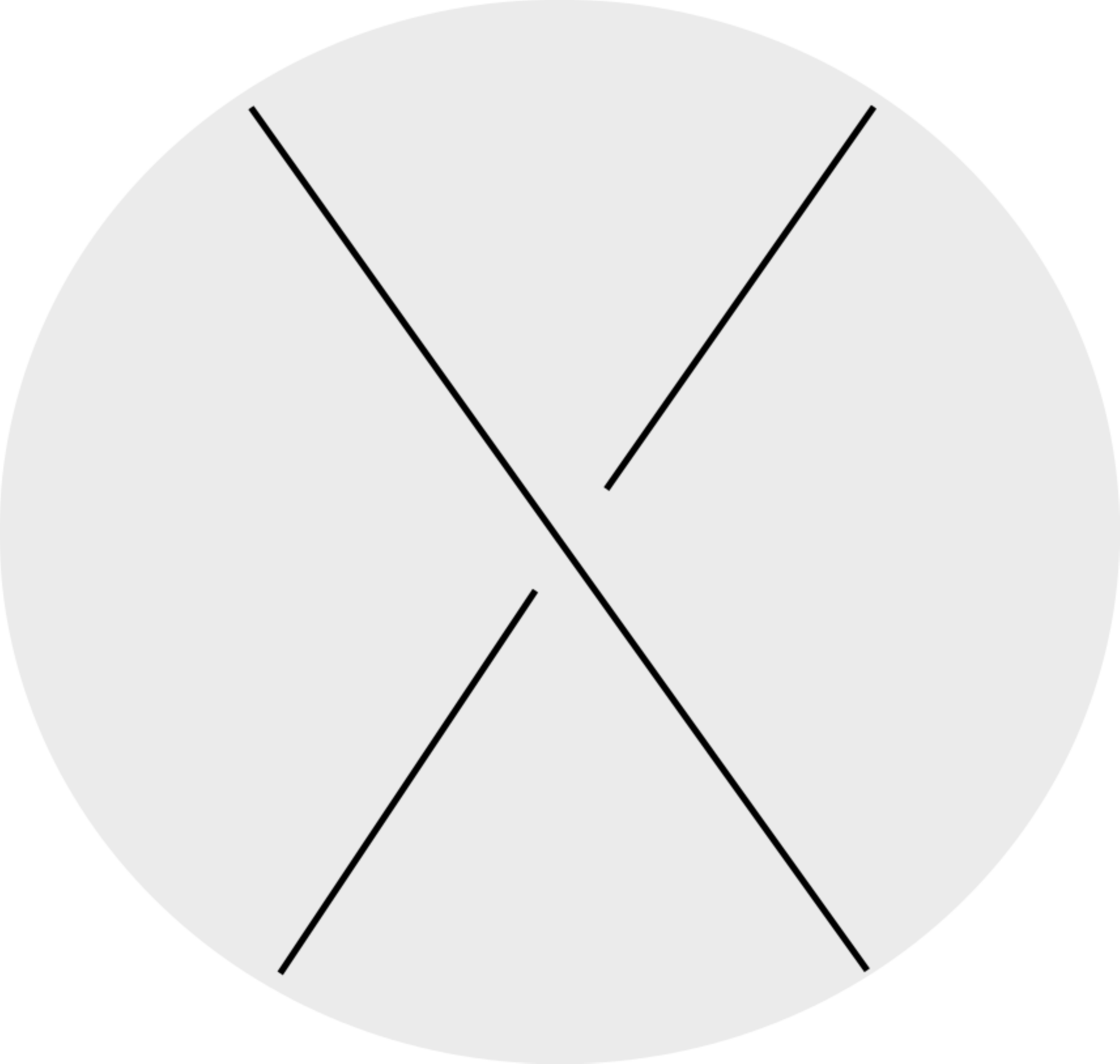}
  \end{tabular}
\caption{switching crossings by gluing an ellipse}
\label{fig:switchingcrossings}
\end{figure}

The following theorem will allow one to change a knot to one whose diagram differs only in the type of crossings.

\begin{theorem}
\label{switchingcrossings}
  If there exists a knot of degree~$d$ with diagram $D_1$, then there exists a
  knot of degree~$d+2$ with a diagram $D_2$ if $D_1$ can be obtained from  $D_2$ by
  changing exactly one under-crossing to an over-crossing or one over-crossing to
  an under-crossing.
\end{theorem}

\begin{proof}
 Consider an ellipse $C$ passing through the point $p_1$, small enough so that it does not contain any other double points of the projection of the knot and so that its projection intersects the projection of the knot in one other point $p_2$ (see Figure~\ref{fig:switchingcrossings}).
\end{proof}

By applying the theorem successively, it is easy to generalize it as follows:
\begin{corollary}
  If there exists two knots $k_1$ and $k_2$ whose diagrams differ by $n$
  crossing changes, and if $k_1$ can be realized as a real rational knot of
  degree~$d$, then $k_2$ can be realized as a real rational knot of degree at
  most $d+2n$.
\end{corollary}

Manturov proved that any knot can be realized as a $(p, q)-$~torus-knot along
with some crossing changes~\cite{manturov2002combinatorial}~\cite[Theorem 16.1]{manturov2004knot}. We therefore have the following corollary:
\begin{corollary}
 Given a knot $k$, consider the $(p,q)-$~torus knot that differs from it by
 crossing changes. If the number of crossing changes is $c$, then there exists a
 real algebraic knot of degree $2(q+c)$.
\end{corollary}
\begin{proof}
 A $(p,q)-$~torus-knot is the link of the singularity associated to a the curve
 $z^p+w^q=0$ at $(0,0)$ in $\mathbb{C}^2$ (see ~\cite[Assertion, Section 1]{milnor2016singular}). Indeed, it is easy to see that it
 is rational and of degree $2q$. Now we apply Manturov's theorem.
\end{proof}

However, in general it may be difficult to find the $(p, q)$-torus knot and the
$p$ and $q$ may be very large. Furthermore, the crossing changes required may be
very large too. In sections \ref{sec:braid} and \ref{sec:polygonal} we will see alternative methods.

\section{Pretzel knots}
 Now we demonstrate a general method of constructing an entire class of knots,
 namely the pretzel knots, with a bound on the degree in terms of the tuple of
 numbers that distinguish the pretzel knots.

\begin{figure}[h]
  \begin{center}
 \includegraphics[width=5cm]{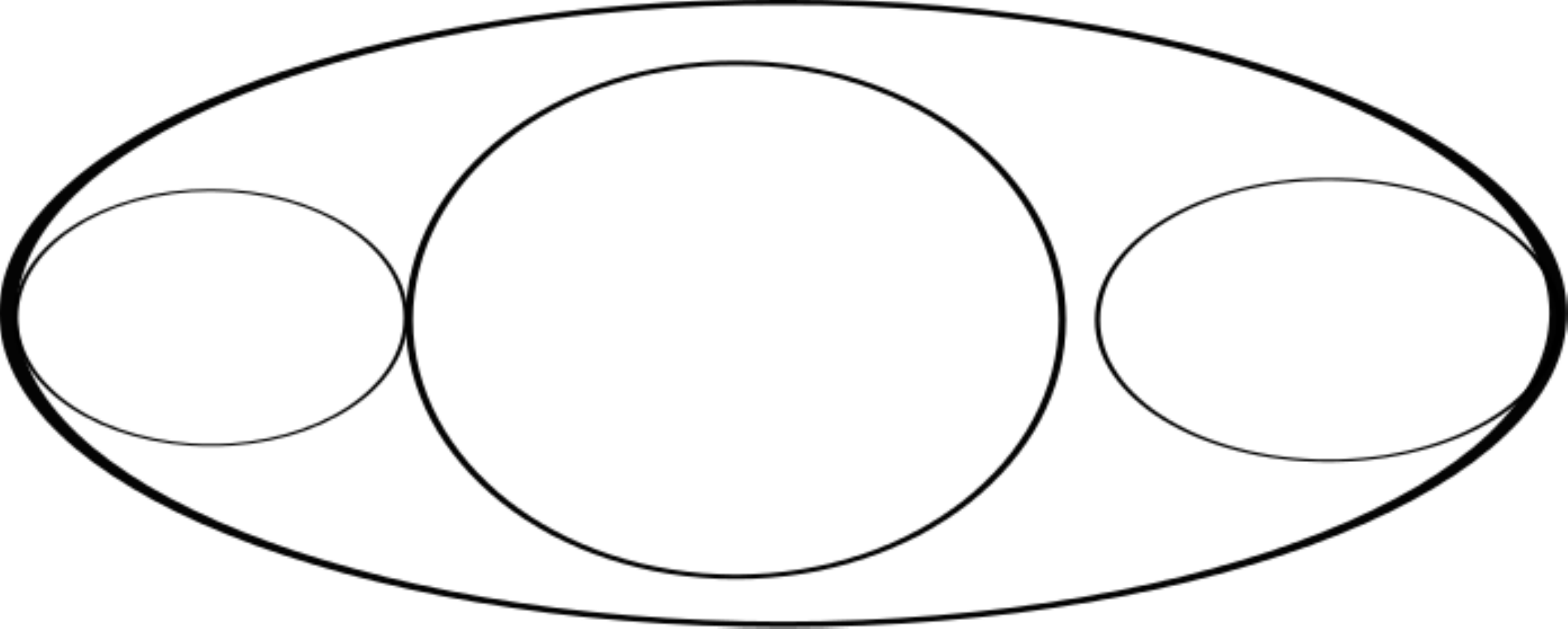} 
 \caption{Skeleton of a pretzel knot}
  \end{center}
 \label{fig:pretzel1}
\end{figure}

\begin{definition}
 The skeleton of a  $(k_1, \ldots, k_n)$-pretzel knot  is the $(\epsilon_1, \ldots, \epsilon_n)$-pretzel knot where $\epsilon_i = k_i\ 
 \mathrm{mod}\ 2$
\end{definition}

\begin{lemma}
 A $(k_1, \ldots, k_n)$-pretzel knot can be constructed by gluing $\sum_i
 [k_i/2]$circles to its skeleton.
\end{lemma}
\begin{proof}
  Figure~\ref{fig:doubletwist}  shows that a double-twist can be added in any
  neighbourhood of the knot where it is isotopic to two parallel lines, by
  gluing a circle. Therefore, by
  gluing $[k_1/2]$ circles one can obtain the necessary twists.
\end{proof}

\begin{figure}[h]
  \begin{center}
 \includegraphics[width=3cm]{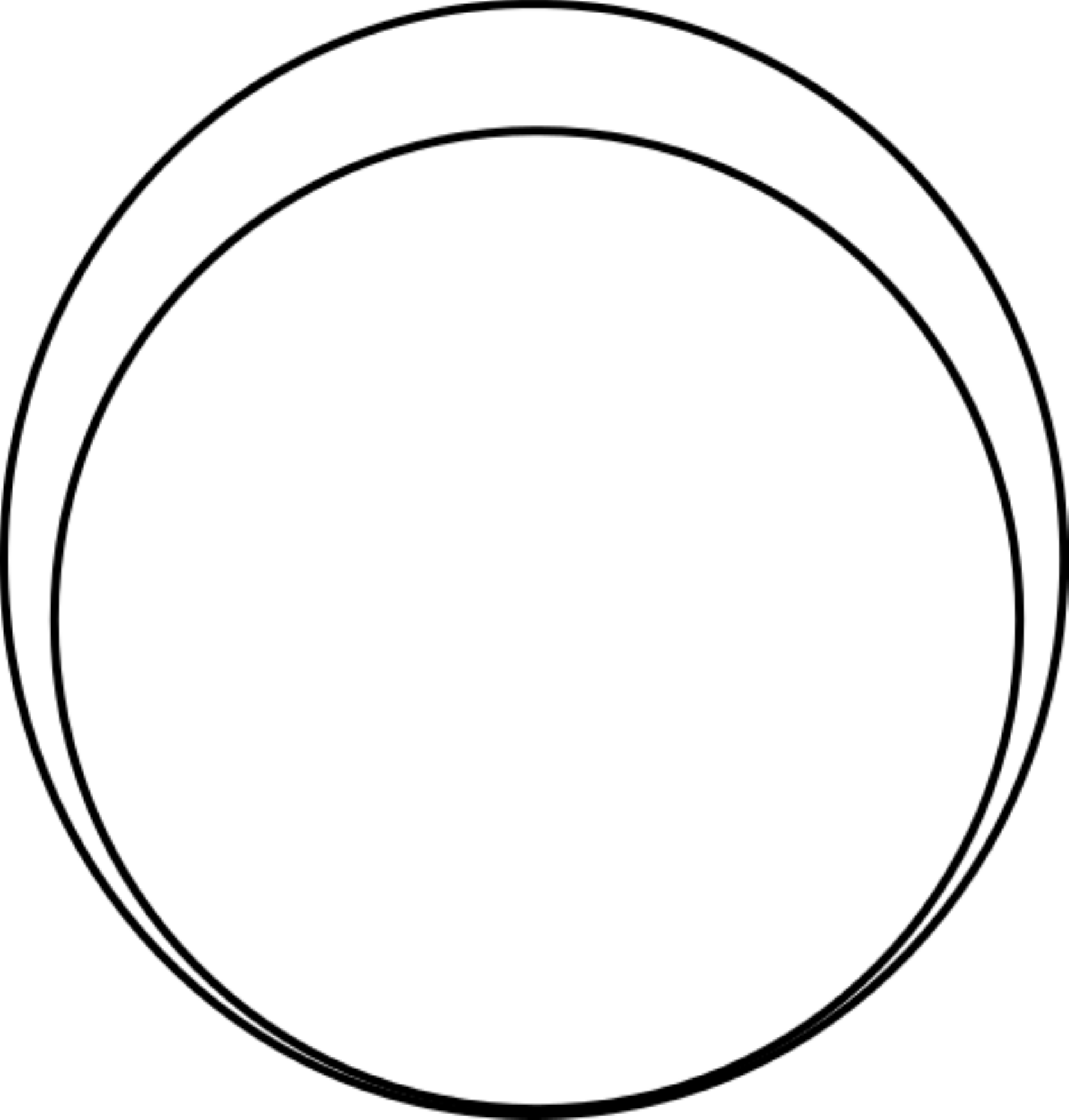} 
 \label{fig:pretorusknot}
 \caption{Skeleton of a torus knot}
  \end{center}
\end{figure}
\begin{lemma}
 A skeleton of a $(k_1,\ldots, k_n)-$pretzel knot can be obtained by gluing $n+1$ ellipses. 
\end{lemma}
\begin{proof}
 Figure~\ref{fig:pretzel1} shows the construction of the skeleton if at least
 one of the $k_i$'s is even. 

Assume that all the $k_i$'s are odd. Then the skeleton is a
$(1,\ldots,1)$-pretzel knot which can easily be seen to be a $(2,k)$-torus knot.
Therefore, if we can construct the $(2,k)$-torus knot, we have only to add
double twists to a strip. However, the $(2,k)$-torus knot itself can be
constructed by performing the gluing perturbation to figure~\ref{fig:pretorusknot}
and then adding the necessary double twists.
\end{proof}

\begin{corollary}
 Given a $(k_1, \ldots, k_n)$-pretzel knot, there exists a real rational knot of
 degree less than or equal to $1+n+\sum_{i=1}^n k_i$
\end{corollary}

In the following sections we will demonstrate general methods of constructing
a real rational representative of any knot, in terms of some numerical invariant
of the knot.
\section{The braid group approach}
\label{sec:braid}
A look at the construction of knots in ~\cite{bjorklund} with small number of
crossings shows that, in general, the diagrams are not necessarily the standard
simple ones. Certain diagrams are more amenable to construction by gluing
ellipses. Even in the example in the previous section, the standard diagram of the skeleton of a $(k_1, \ldots, k_n)$-pretzel knot, where all the $k_i$ are odd, had to be changed to the standard diagram of a $(2,k)$-torus knot. Our goal is to find a general procedure of constructing all possible knots, and therefore a general method of finding a suitable diagram. Furthermore we wish to be able to obtain a reasonable bound on the degree.

It is well known that any knot can be represented as the closure of a braid.
Here, we will see that the group structure of the braid group (more
specifically, the isomorphism $B_n/P_n \cong S_n$) helps us in
finding a representation of any given knot that can be constructed by repeatedly
applying two very simple gluing moves that are shown in
figure~\ref{fig:braidmove}. The bound will be in terms of the minimum number of
certain types of generators required to express the braid associated to the
knot. In order to fix some notation, we recall the following simple
theorem: 

\begin{theorem}
  \label{purebraidisomorphism}
There is a natural group homomorphism, $\theta : B_n \to S_n$, from
the Braid group $B_n$ with $n$ strands to the permutation group $S_n$ whose kernel is the group of pure braids $P_n$.
\end{theorem}

Placing $n$ vertical ellipses stacked up in parallel can be interpreted as the
closure of the identity braid in $B_n$. We denote this trivial braid
by $I_n$. We can construct non-identity braids of $n$ strands by the following two two possible types of gluing, which will prove sufficient for our purpose:

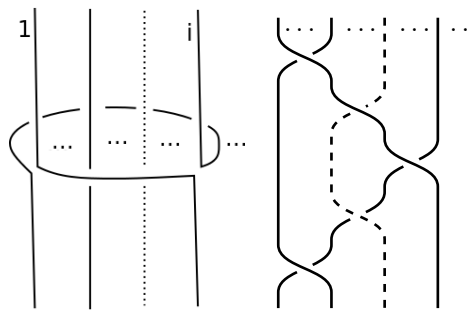
\begin{figure}[h]
  \begin{center}
  \begin{tabular}{ccc}
   \includegraphics[height=4cm]{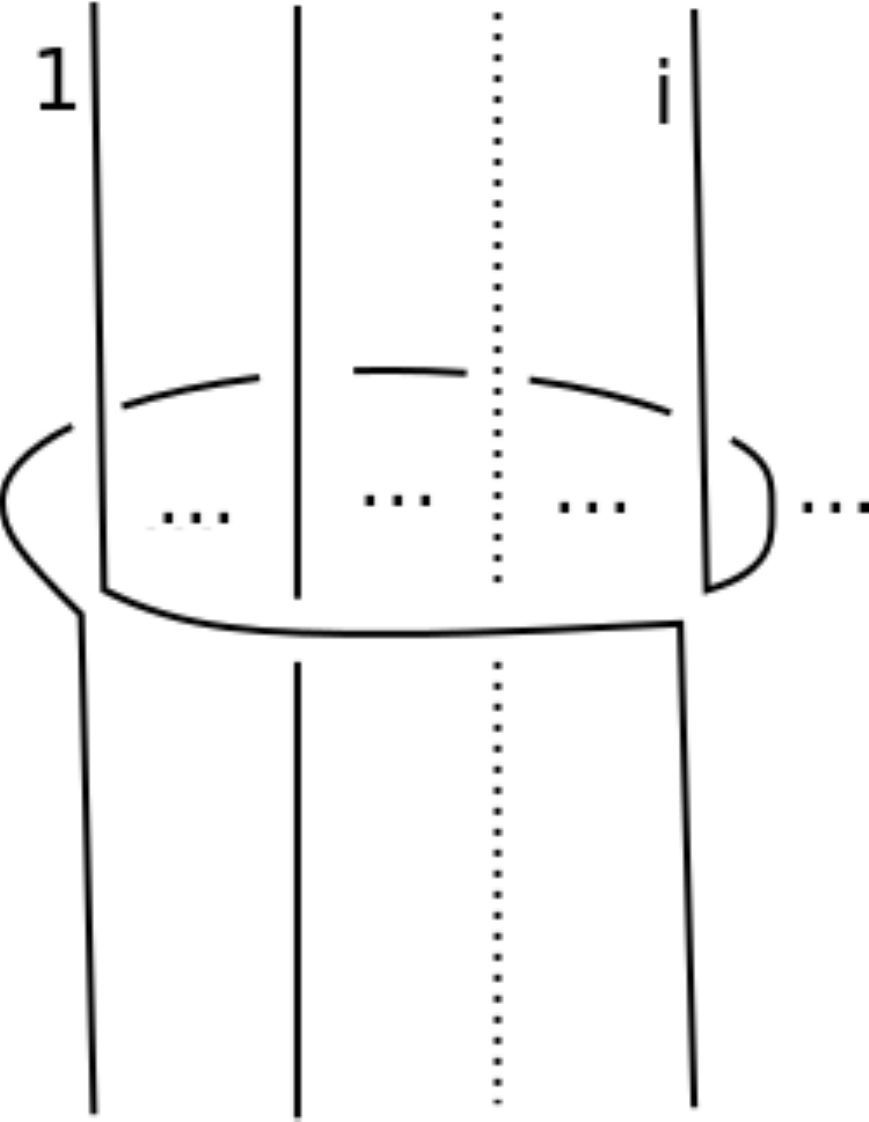}& 
 \begin{picture}(60,0)
  \begin{tikzpicture}[scale=0.7]
   \braid[line width=1, style strands={3}{dashed}] a_1 a_2 a_3 a_2 a_1; 
  \end{tikzpicture}
\put(4,105){\ldots}
\put(-15,105){\ldots}
\put(-35,105){\ldots}
\put(-58,105){\ldots}
 \end{picture} 
  \end{tabular}
  \end{center}
\caption{
  A rectangular neighbourhood of the horizontal ellipse corresponding to the
  word
  $W_i=\sigma_1\sigma_2\ldots\sigma_{i-2}\sigma_{i-1}\sigma_{i-2}\ldots\sigma_2\sigma_1$
  (the braid form is shown on the right). It connects the 1st ellipse with the
  $i$th ellipse via gluing (1st and $i$th vertical ellipse form one component).
  The image under $\theta$ is a transposition. Note that the vertical lines are
  parts of large vertical parallel ellipses. The dashed lines represent ellipses
  that are yet to be constructed and connected with the first ellipse. The
  (second) non-dashed line represents an ellipse that has been connected with
  the first one by a horizontal ellipse, above the given one, not visible in this picture.}
\end{figure}

\begin{figure}[h]
  \begin{center}
  \begin{tabular}{ccc}
   \includegraphics[height=4.5cm]{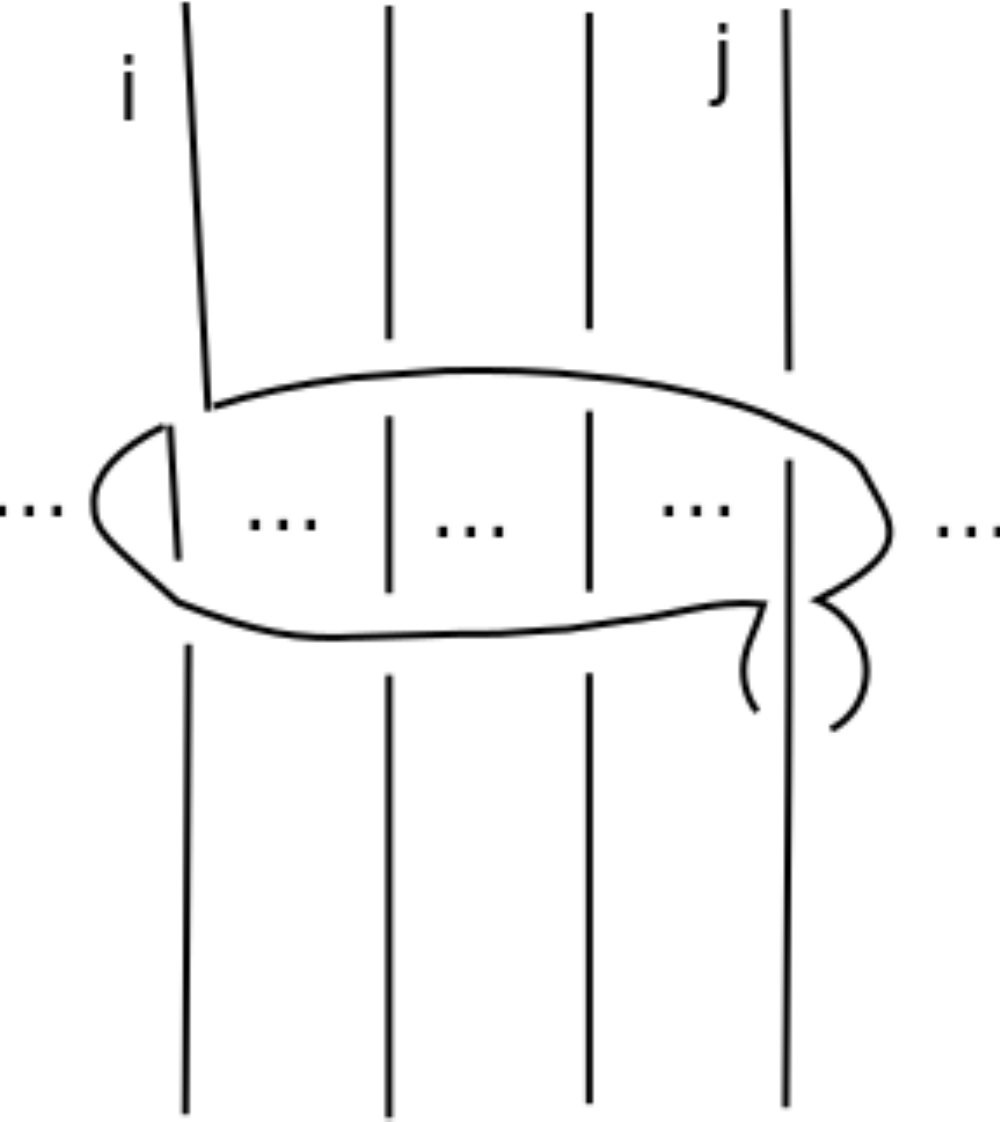}&
 \begin{picture}(60,-50)
  \begin{tikzpicture}[scale=0.7]
   \braid[line width=1] a_1 a_2 a_3 a_3 a_2^{-1} a_1^{-1}; 
  \end{tikzpicture}
\put(4,125){\ldots}
\put(-15,125){\ldots}
\put(-35,125){\ldots}
\put(-58,125){\ldots}
\put(-78,125){\ldots}
 \end{picture} 
 \\
 
  \end{tabular}
  \end{center}
\caption{A rectangular neighbourhood of the horizontal ellipse corresponding to
  the word $B_{i.j}=\sigma_i\sigma_{i+1}\ldots\sigma_{j-1}^2\ldots \sigma_i$,
  which is a generator of the pure braid group. The $n$ vertical lines are parts
  of large parallel ellipses and have already been joined together by horizontal
  ellipses corresponding to the words $W_i$ (gluing moves of type 1).} 
  %\caption{Two gluing moves}
  \label{fig:braidmove}
\end{figure}
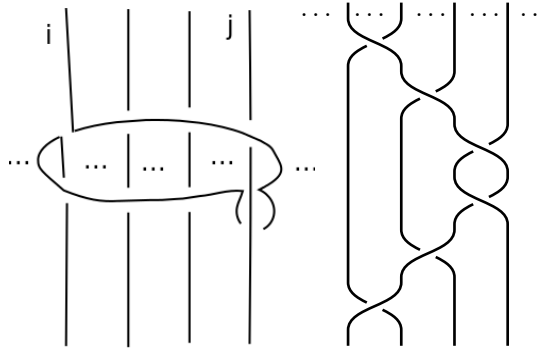

\begin{figure}[h]
\begin{center}
 \begin{tabular}{cc} 
  \includegraphics[width=4cm]{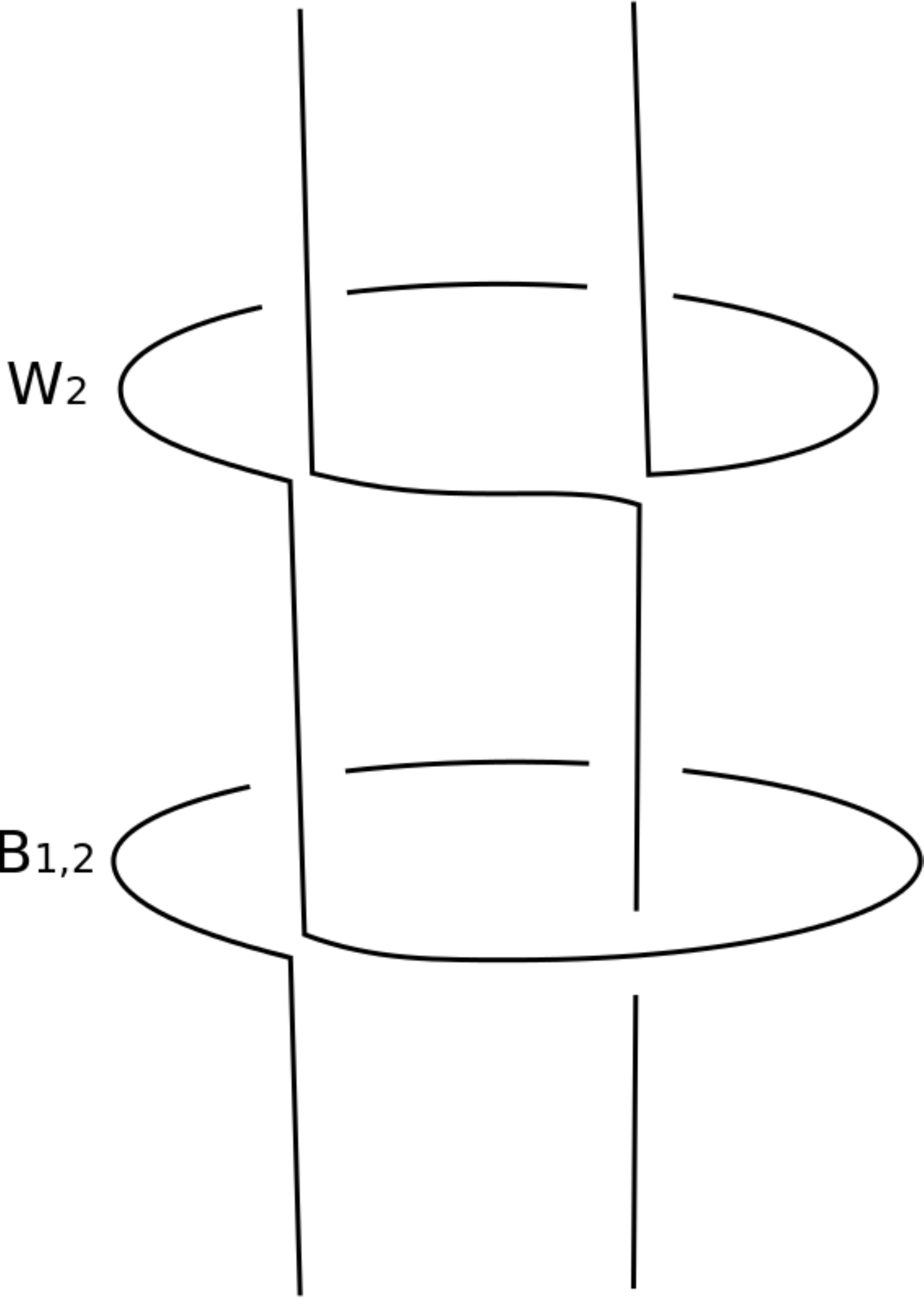} &
  \includegraphics[width=5cm]{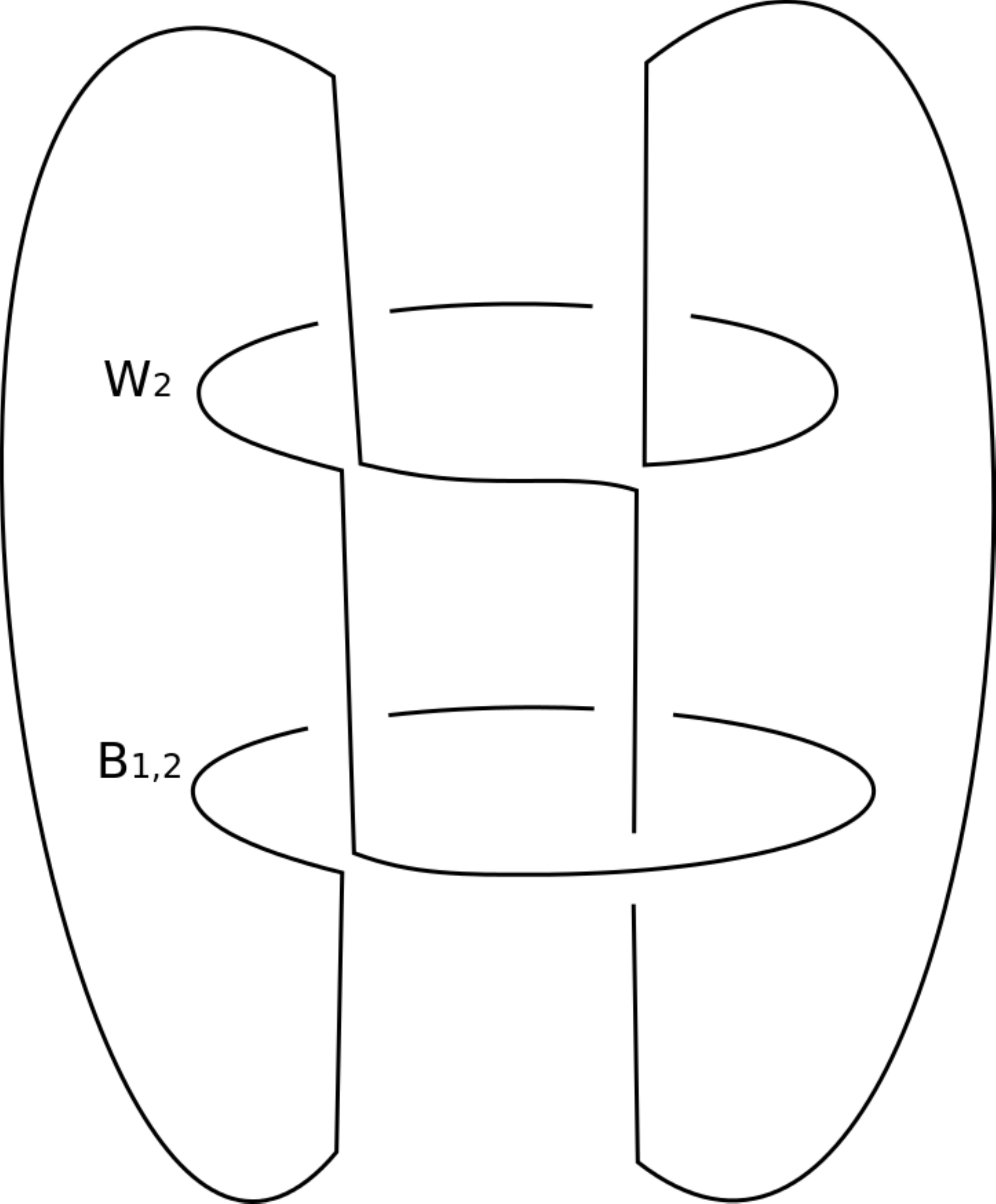} \\
The closure of the trefoil & With the closure strands visible (and distorted)
 \end{tabular}
\end{center}  
\caption{Constructing the trefoil out of the word representation $W_2B_{1,2}$.
  The vertical lines in the first picture are parts of vertical ellipses. In the
second picture, the rest of the ellipse is visible and distorted to fit it in
the picture}
  \label{fig:braidtrefoil}
\end{figure}

\begin{figure}[h]
\begin{center}
 \begin{tabular}{cc} 
  \includegraphics[width=3cm]{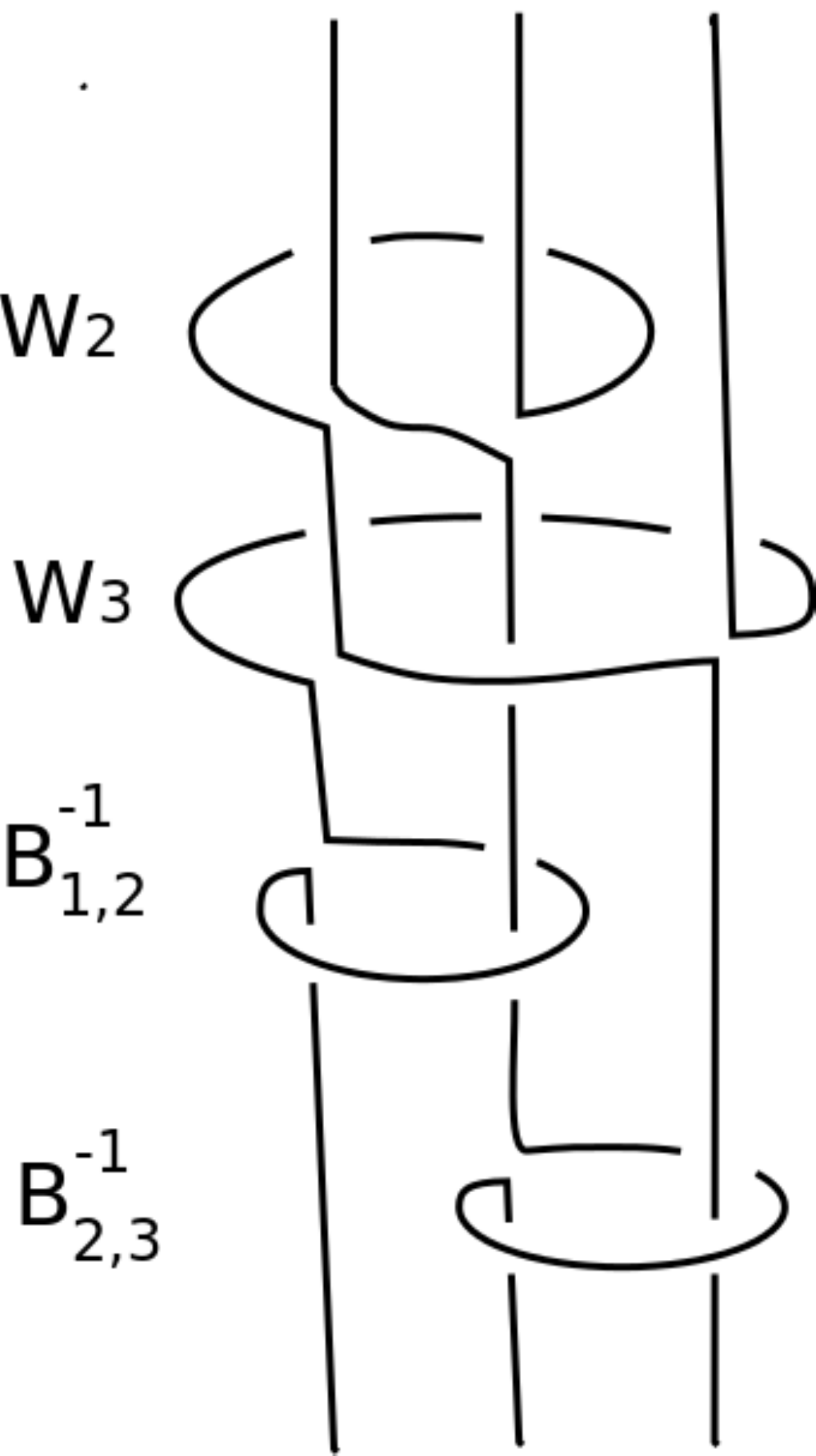} &
  \includegraphics[width=5cm]{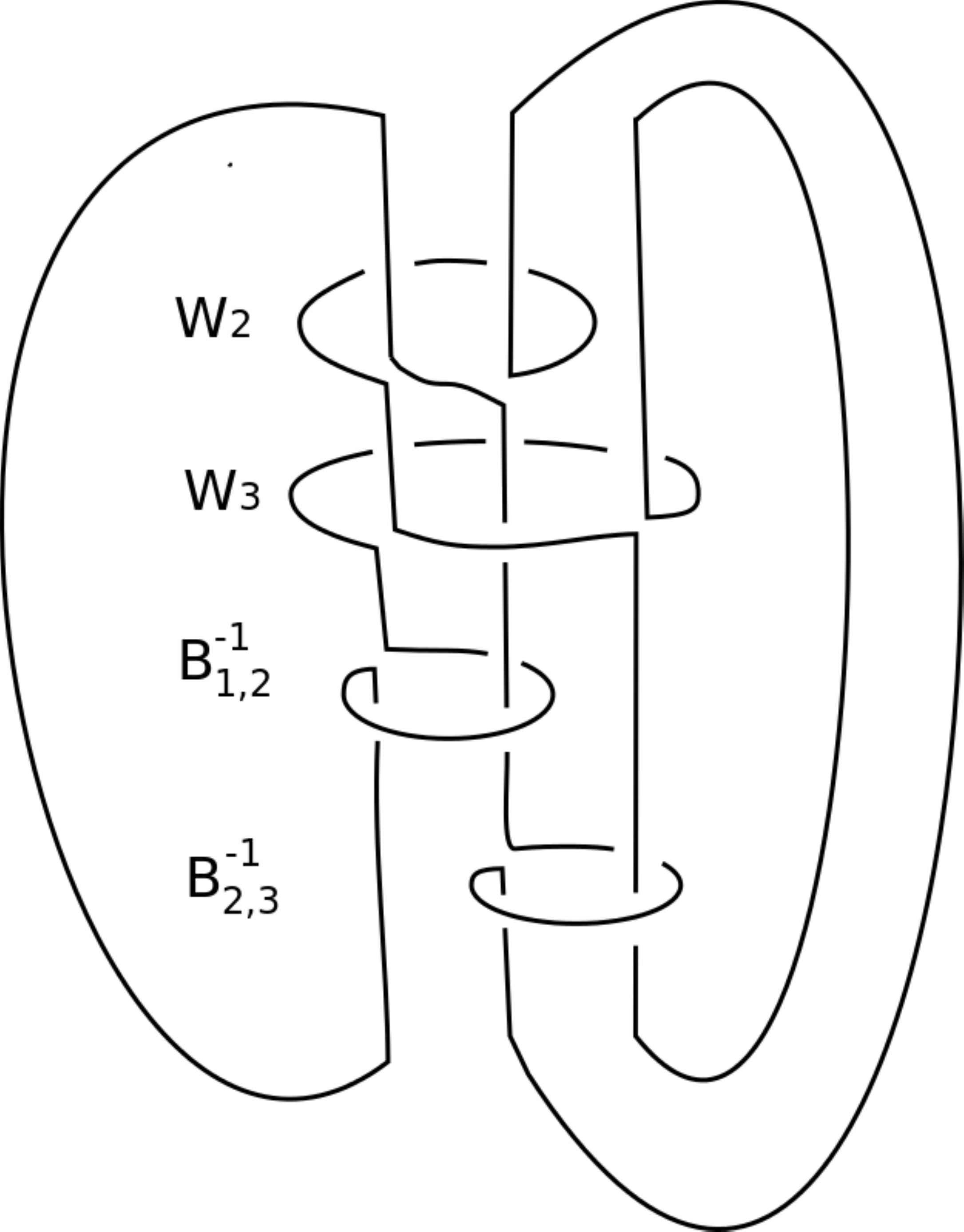} \\
The closure of the figure eight & With the closure strands visible (and distorted)
 \end{tabular}
\end{center}  
\caption{constructing the figure eight out of the word representation $W_2W_3B_{1,2}^{-1}B_{2,3}^{-1}$}
  \label{fig:braidfigureeight}
\end{figure}
\begin{description}
 \item [Gluing move of type 1] As shown in the first image of
   figure~\ref{fig:braidmove}, we may introduce a new strand by gluing a
   horizontal ellipse to the first vertical ellipses and then \textit{again} to
   the new vertical ellipse. We call this ellipse the \textit{horizontal ellipse
   corresponding to the word $W_i$}. Note that we cannot glue more than $n-1$ horizontal
   ellipses in this way, because it would result in single component. Note that
   in figure~\ref{fig:braidmove}, the vertical lines are parts of ellipses. The lines in bold
   are the ones that have already been glued by this method. The dashed lines
   are parts of the vertical ellipses that are yet to be glued. 
   \item [Gluing move of type 2] As shown in the second image of
     figure~\ref{fig:braidmove}, we may glue a horizontal ellipse at only one
     point to any vertical ellipse, say the $i$th. For reasons that will be
     clear later, we switch the last crossing to surround the $j$th ellipse. We
     call this ellipse the \textit{horizontal ellipse corresponding to the word
       $B_{i,j}$}. This does not decrease the number of connected components, and therefore any number of horizontal ellipses may be glued in this way.
\end{description}

It is easy to see that if one isotopes the strands so that they are decreasing,
as is required in the geometric definition of a  braid, the word introduced by
the first gluing move is of the type in the following definition:
\begin{definition}
 We define the $i$th transposed word, denoted by $W_i$, to be the word $\sigma_1\sigma_2\ldots\sigma_{i-2}\sigma_{i-1}\sigma_{i-2}\ldots\sigma_2\sigma_1$
\end{definition}

Note that $\theta(W_i)$ is the transposition $(1\ \ i-1)$.

\begin{lemma}
 It is possible to construct the closure of the $n$-braid 
\[W_{i_1} W_{i_2}\ldots W_{i_{n-1}} \in B_n\]  
as long as $i_1, i_2, \ldots, i_m$ are all distinct, by applying $m$ moves of
type gluing move 1
  \label{permutationWord}
\end{lemma}
\begin{proof}
  We will proceed by induction
 on the number of strands of the braid. Assume we have already glued horizontal
 ellipses corresponding to the words $W_{i_1},\ldots, W_{i_k}$. Below
 all the already glued horizontal ellipses, glue a horizontal ellipse to the
 first ellipse so that its other end passes through a point that will be
 contained in the $i_{k+1}$th ellipse (presently omitted). Observe that since $i_{k+1} \neq i_j$ for $j \leq k$, the
 $i_{k+1}$th ellipse is a distinct component from what is obtained by gluing
 horizontal ellipses corresponding to the words $W_{i_1},\ldots, W_{i_k}$ along with the horizontal ellipse that has just been
 glued. Therefore we can glue the $i_{k+1}$th ellipse also to the horizontal
 circle as shown in figure~\ref{fig:braidmove}. Once we have glued all the
 $n$ vertical ellipses, we would have obtained the closure of a braid that
 corresponds to the given word.
\end{proof}

The point of the lemma is this corollary (recall that the image under $\theta$ of a braid is a cycle if and only if its
closure is a knot i.e. a one component link):

\begin{corollary}
  \label{permutation}
 Given any cycle $\tau \in S_n$ there exists a fixed glued real rational knot
 which is the closure of a braid in the coset $\theta^{-1}(x)$. For a given
 permutation $\tau$, we denote this braid by $\beta_{\tau}$
\end{corollary}
\begin{proof}
Represent the cycle $\tau =(i_1\ i_1\ \ldots\ i_n)$  so that $i_1=1$. Then this
cycle can be written as the product of transpositions like this: $(1\ i_2)(1\ 
i_3)\ldots (1\  i_n)$. However, this is precisely $\theta(W_{i_2}W_{i_2}\ldots
W_{i_n})$ where $i_1, i_2, \ldots, i_n$ are all distinct.
\end{proof}

The above corollary has produced one glued closure of a braid for each possible
cyclic permutation. Owing to theorem~\ref{purebraidisomorphism}, the union of
the cosets $\beta_{\tau}P_n$, where $\tau$ ranges over all possible cyclic
permutations, is the entire braid group. Therefore, we only have to prove that
we can append the above words with words that represent all the possible pure
braids. It is well known~\cite{birman2016braids, manturov2004knot} that the pure braid group can be generated by elements
of the following type:

\begin{definition}
 A word is called a standard pure braid generator, denoted by $B_{i, j}$, if it
 is of the form $\sigma_i\sigma_{i+1}\ldots\sigma_{j-1}^2\ldots \sigma_i$.
\end{definition}

$B_{i.j}$ is precisely a gluing move of type 2, and therefore:
  
\begin{lemma}
 It is possible to construct the closure of the $n$-braid 
\[W_{i_1} W_{i_2}\ldots W_{i_{n-1}} B_{j_1, k_1}B_{j_2, k_2}\ldots B_{j_l, k_l} \in B_n\]
 as long as $i_1, i_2, \ldots, i_m$ are all distinct, by performing $l$ gluing
 moves of type 2.
\end{lemma}
\begin{proof}
 We already know by lemma~\ref{permutationWord} that we can construct the closure of $W_{i_1} W_{i_2}\ldots W_{i_{n-1}}$ by gluing $n$ ellipses. We will prove the rest by induction. Assume that we can construct $W_{i_1} W_{i_2}\ldots W_{i_{n-1}} B_{j_1, k_1}B_{j_2, k_2}\ldots B_{j_r, k_r}$ by gluing at most $n-1+2r$ horizontal ellipses, then we can glue an ellipse as shown in the second part of figure~\ref{fig:braidmove} and another to switch the crossing. Observe that there are no restrictions on the number of such ellipses that we glue because one is gluing it only at one point and therefore never encountering a self gluing.
\end{proof}

Note that a word $B_{i,i+1}$ does not require a switch of crossings and
therefore requires the gluing of a single ellipse rather than two. It is easy to
keep a track of the number of ellipses that are glued and therefore we can
collect all the above in our main theorem:

\begin{theorem}
 Given an $n$-braid $b$, such that $\theta(b)$ is the cycle represented by $(1\ i_1\ i_2\ \ldots\ i_n)$, the pure braid $W_{i_n}^{-1}W_{i_{n-1}}^{-1}\ldots W_{i_1}^{-1}b$ can be expressed as the product of the standard generators
 $B_{i,j}$. If it is the product of $l$ such generators, out of which $r$ of
 them are of the form $B_{i, i+1}$ then there exists a real algebraic knot of degree $4n+4l-2r-2$, which can be constructed by gluing ellipses. 
\end{theorem}

The above method reduces the the problem of constructing a real rational
representative of a given knot to the purely algebraic problem of finding its
decomposition in terms of the generators $W_i$ and $B_{j,k}$, which we have
proved is always possible. The method is algorithmic, and we use it to construct
the trefoil and figure-eight.

\begin{example}
 The trefoil is the closure of the 2-braid $\sigma_1^3$. $\theta(\sigma_1^3)=
 (1\ 2)$. and $W_2^{-1}\sigma_1^3= \sigma_1^2=B_{1,2}$. Therefore, there exists a
 real rational trefoil (the closure of $W_2B_{1,2}$) of degree~8 as shown in figure~\ref{fig:braidtrefoil}.
\end{example}

\begin{example}
 The figure-eight knot is the closure of the 3-braid
 $\sigma_1\sigma_2^{-1}\sigma_1\sigma_2^{-1}$. Its image under $\theta$  is
 $(1\ 2\ 3)$ and it can be checked that
 $W_3^{-1}W_2^{-1}\sigma_1\sigma_2^{-1}\sigma_1\sigma_2^{-1}=B_{1,2}B_{2,3}$.
 Therefore, there exists a real rational representation of the figure-eight
 knot (the closure of $W_2W_3B_{1,2}B_{2,3}$) of degree 14 as show in the figure~\ref{fig:braidfigureeight}
\end{example}

\begin{remark}
 It is natural to ask if the group operation on braids, which is obtained by
 concatenation, can be obtained by gluing closures. It can be done only for pure
 braids because otherwise it would involve self gluing. Nevertheless, since the
 pure braid subgroup is a normal subgroup of the braid group, and the quotient
 is the permutation group so the problem is reduced to being able to fixing
 words representing a braid for each permutation that can be appended to the
 given braid by gluing.  
\end{remark}

\begin{remark}
 For simplicity, we restricted our attention to knots. However, the same method
 can be used to construct all possible real rational links, defined as the union
 of real rational knots.
\end{remark}

\section{Projective to affine}

\begin{definition}
 An affine knot is a projective knot that lies in the affine part of $\mathbb{RP}^3$ 
\end{definition}

Before discussing the other methods, let us first deal with the problem of
converting a projective real rational knot to an affine real rational knot by
gluing. To do so we will need the following lemma:
\begin{lemma}
 Given a real algebraic knot $k$ of degree~$d$, that intersects the plane at
 infinity in $n$ points, there exists a real algebraic knot of degree $d+1$
 isotopic to the given knot, which intersects the plane at infinity in $n-1$ points.
 \label{lem:affine}
\end{lemma}
\begin{proof}
  The knot of degree ~$d+1$ will be constructed by gluing a line $L$ at the point of intersection of $k$ with the plane at infinity. By the gluing lemma, we know that for small enough $\epsilon_1$ and $\epsilon_2$, $G_{s, t}(k, L)$ is the glued knot for all $s <\epsilon_1$ and $t< \epsilon_2$. Denote the point of intersection of the line with the knot by $p$. We will prove that there exist $\epsilon_1'$ and $\epsilon_2'$ such that for all $s<\epsilon_1'$ and $t<\epsilon_2'$, the complexification of $G_{s, t}(k, L)$ intersects the plane at infinity in an imaginary pair in a small neighbourhood around $p$. 
  
Choose (affine) coordinates so that $p$ is given coordinates $(0, 0, 0)$ and the plane at infinity is defined by $x_1=0$. Choose a parametrization of the line which takes 0 to $(0, 0, 0)$. Such a
parametrization would be of the form $t \to (at, bt, ct)$. Let the
paramatrization of the knot be $(p_1/p_0, p_2/p_0, p_3/p_0)$ for some
polynomials $p_i$ so that $p_i(0) = 0$ for $i=1, 2, 3$.  Then $G_{\lambda_1, \lambda_2}(k, L) =  (p_1(\lambda_1t)/p_0(\lambda_1t),
p_2(\lambda_1t)/p_0(\lambda_1t), p_3(\lambda_1t)/p_0(\lambda_1t)) +
(\lambda_1a/t, \lambda_2b/t, \lambda_3c/t)$. This would intersect the plane at infinity whenever $tp_0(\lambda_1t) + \lambda_2ap_0(\lambda_1t) = 0$. Note that for $\lambda_2=0$, this is merely $tp_0(\lambda_1t)=0$ and since we chose a parametrization which takes 0 to $(0, 0, 0)$, 0 is a double root of $tp_0(\lambda_1t)$ and therefore $tp_0(\lambda_1t)$ has $n+1$ roots. Therefore, for a small perturbation, depending on whether $\lambda_2$ is positive or negative,  the double root 0 of $tp_0(\lambda_1t) + \lambda_2ap_0(\lambda_1t)$ will change to either a pair of real roots or a pair of conjugate imaginary roots. Choose the perturbation which results in a pair of imaginary roots. Since the perturbation is small enough, the rest of the real roots remain real and the conjugate imaginary pairs remain imaginary. Therefore,
$tp_0(\lambda_1t) + \lambda_2ap_0(\lambda_1t)$ has $n-1$ real roots and the
result follows. 
\end{proof}

\begin{figure}
  \begin{tabular}{cc}
   \includegraphics[width=6cm]{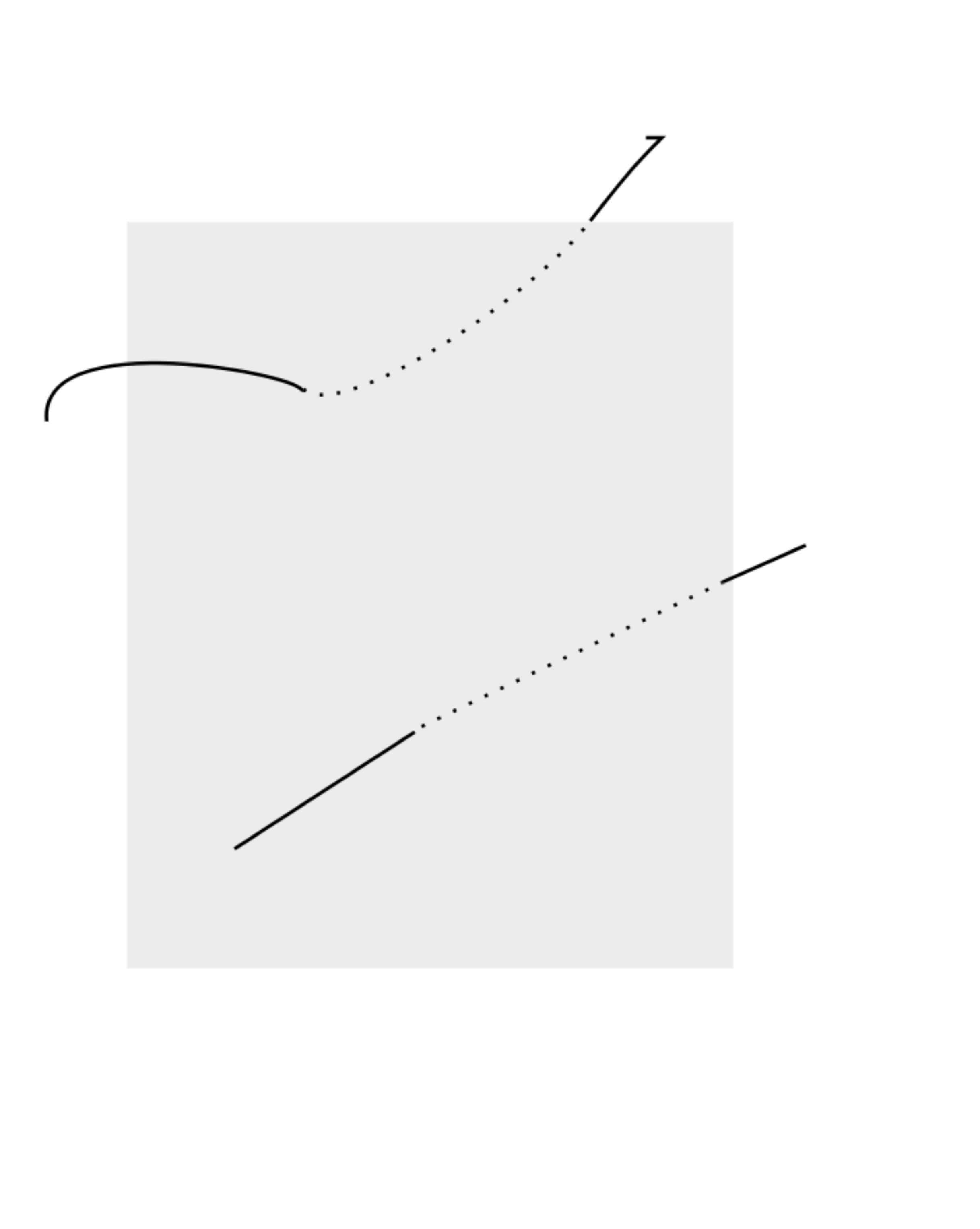}&
   \includegraphics[width=6cm]{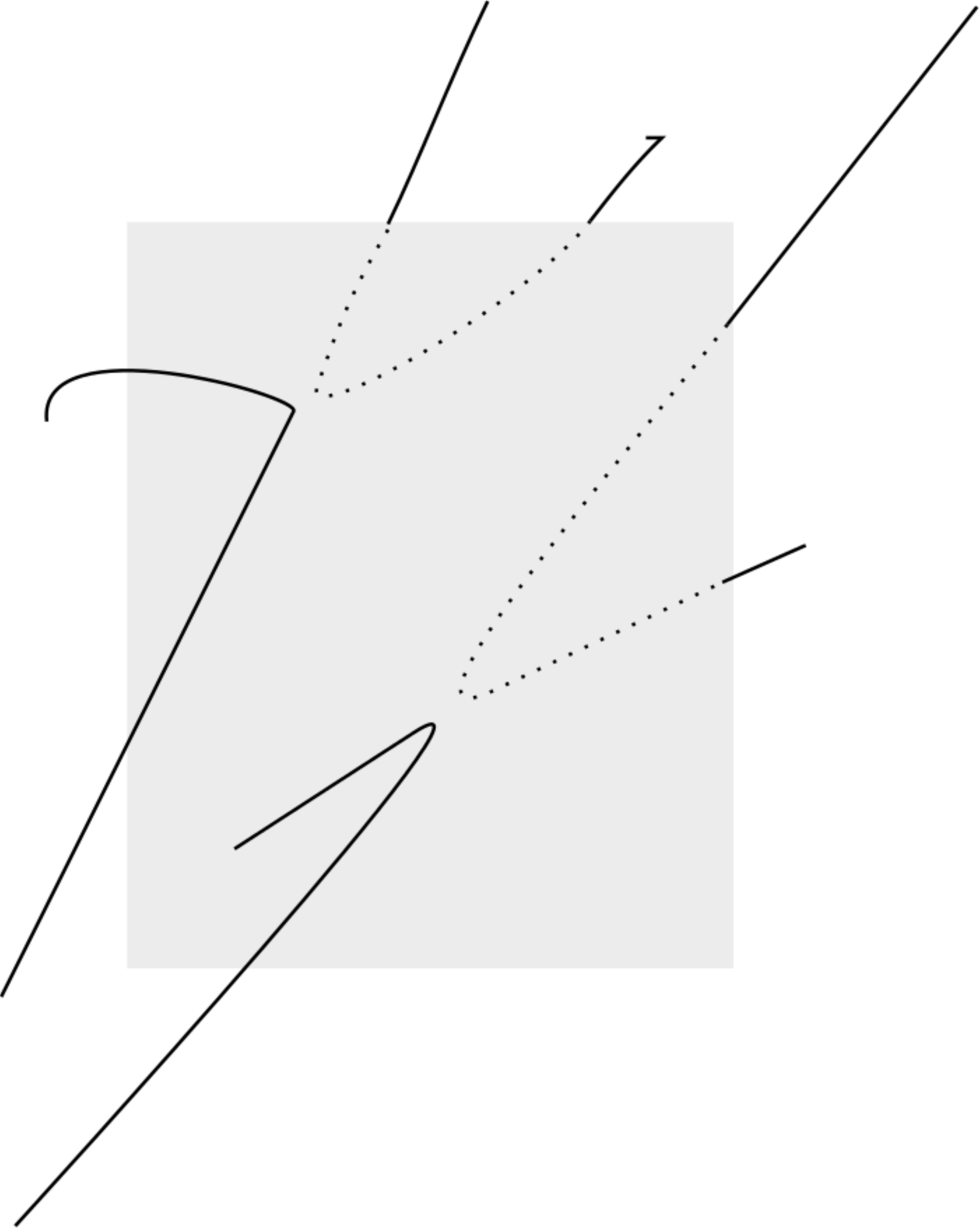}
  \end{tabular}
  \caption{Converting a non-affine curve to an affine curve by gluing lines}
  \label{fig:nonaffinetoaffine}
\end{figure}
\begin{theorem}
  If a plane intersects a knot of degree~$d$ in $k$ points, then there exists an
  affine version of the knot which can be obtained by gluing $k$ lines and
  therefore increases the degree by $k$.
\end{theorem}
\begin{proof}
Repeated application of lemma~\ref{lem:affine} ensures that
that the knot does not intersect the plane. Choose coordinates so that the plane is the plane at infinity.
\end{proof}

\section{The Polygonal Approach}
\label{sec:polygonal}
\begin{figure}[h]
\begin{center}
  \includegraphics[width=6cm]{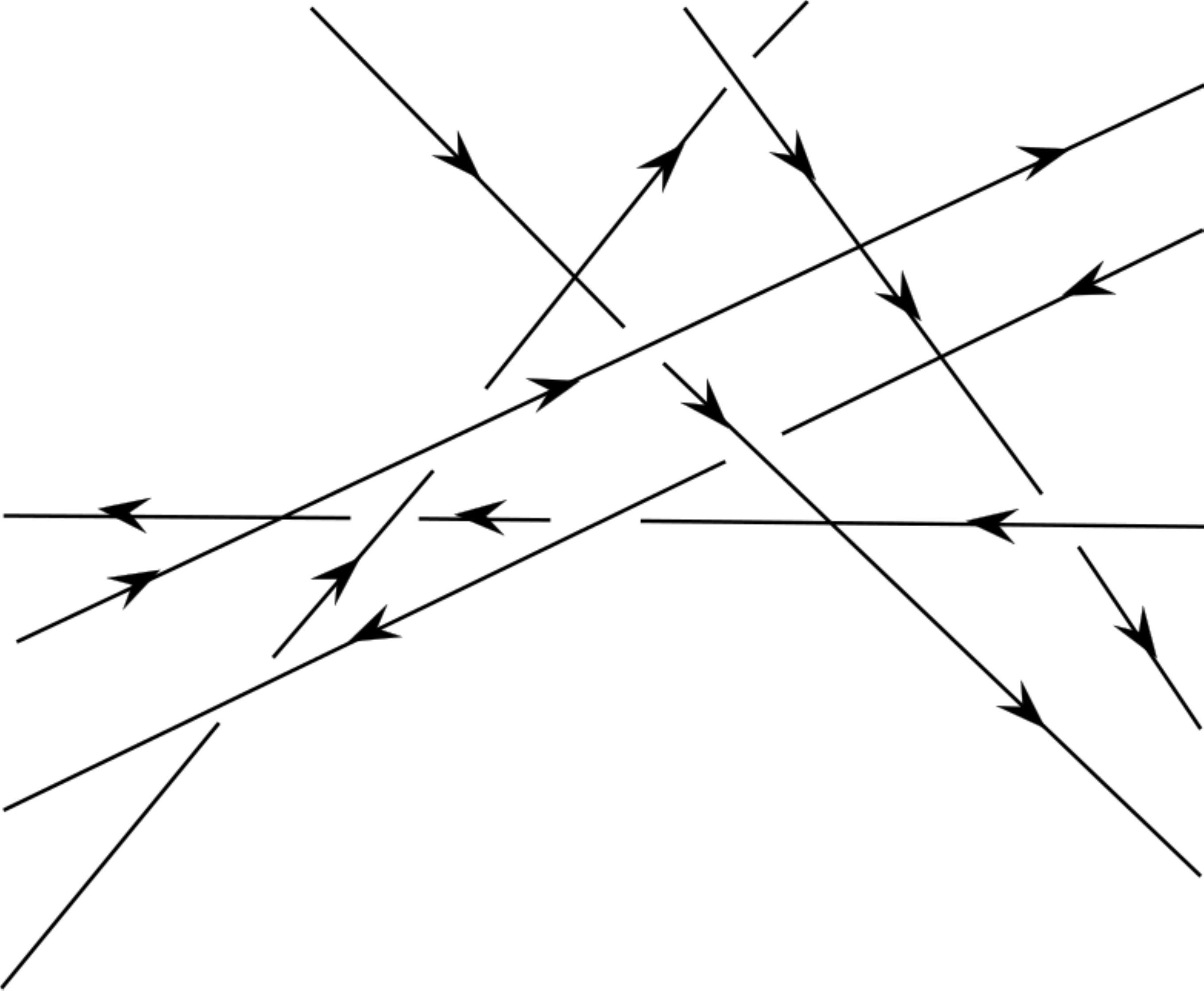}  
\begin{tabular}{cc}
  \includegraphics[width=6cm]{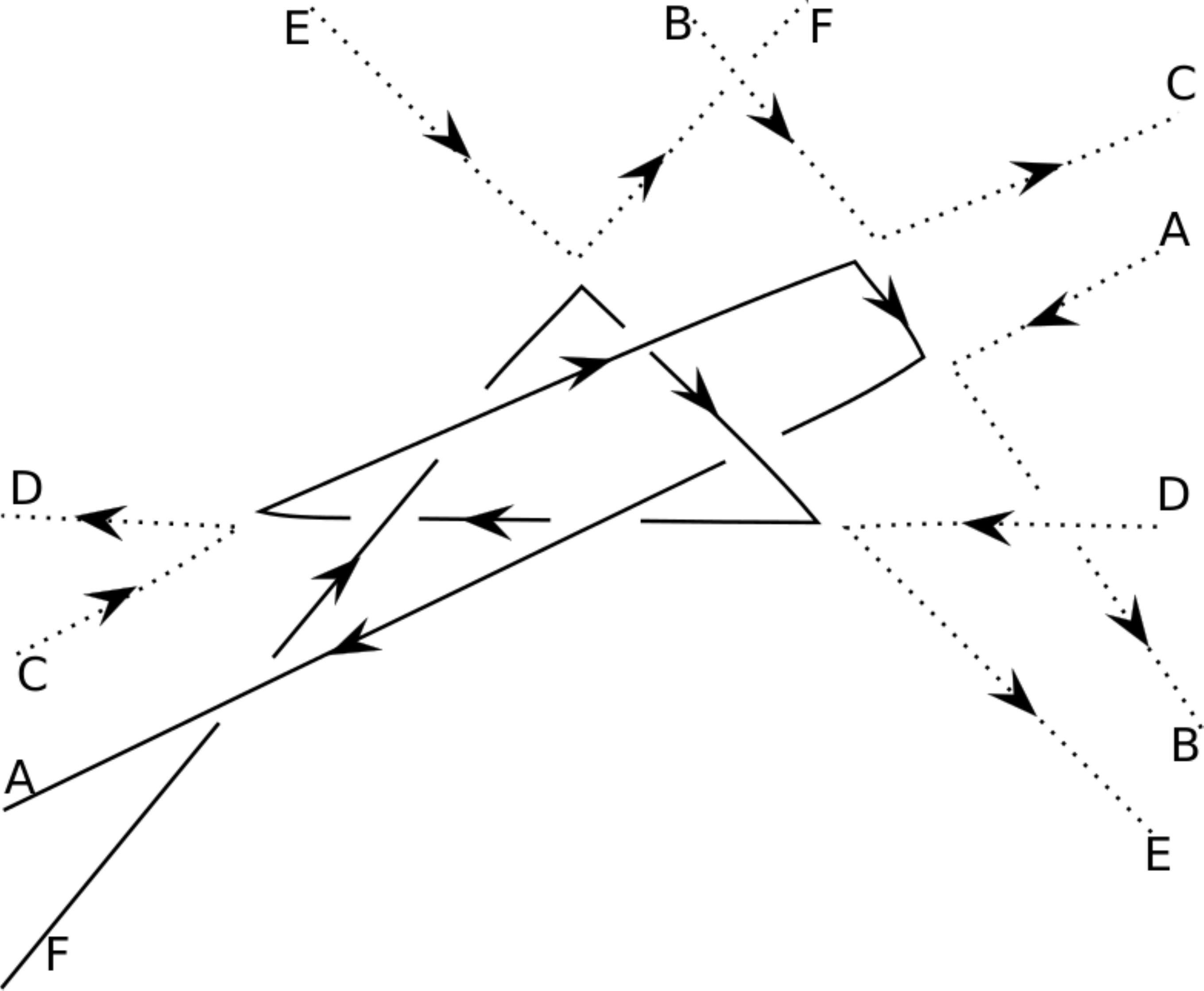} & 
  \includegraphics[width=6cm]{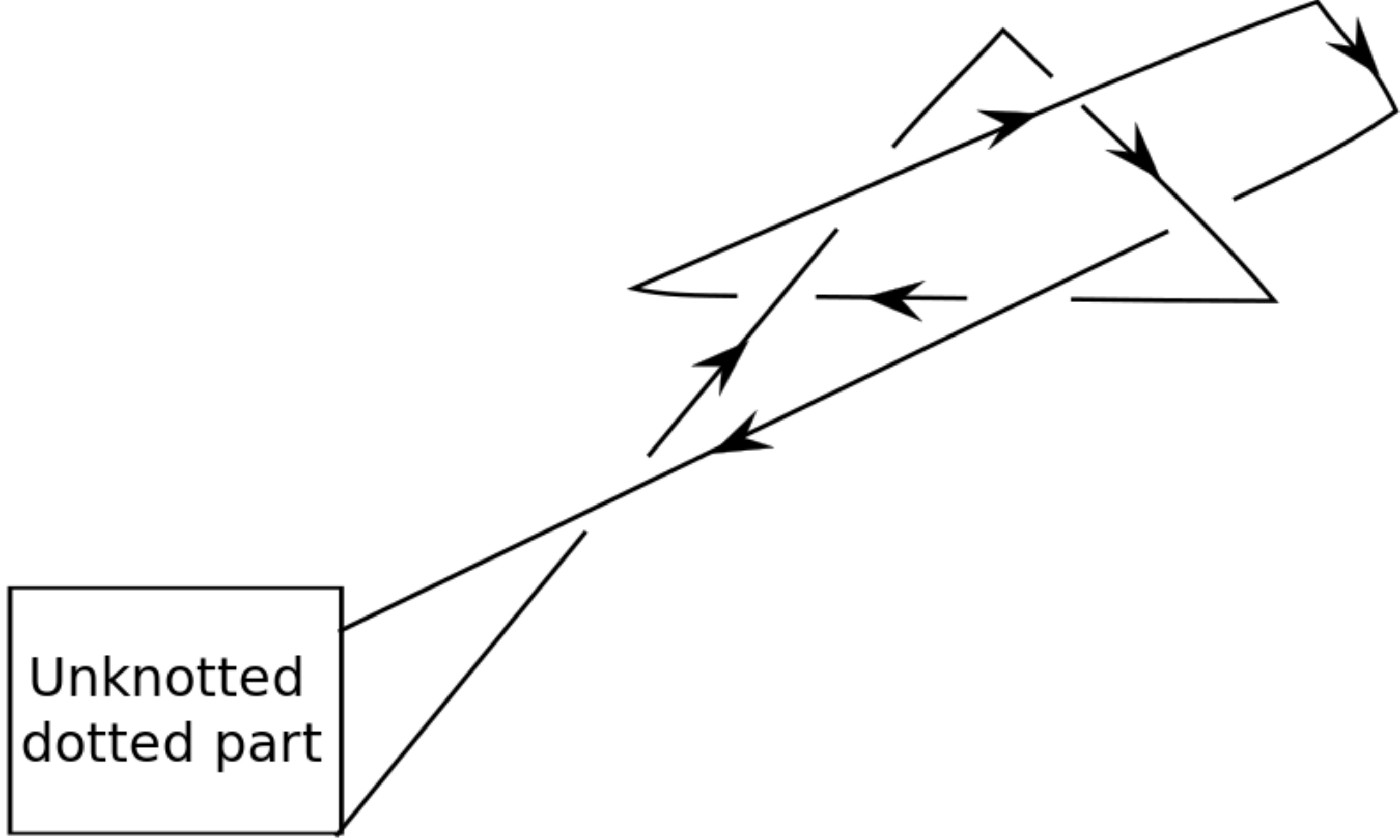}\\
After gluing lines & Unknotting the dotted part 
\end{tabular}
\end{center}
  \caption{A polygonal trefoil formed by gluing lines}
\label{fig:polygonaltrefoil}
\end{figure}

\subsubsection{Using lines}
In the following lemma, by a tangle we will mean an embedding in $\mathbb{RP}^3$ of the interval $[0,1]$

\begin{lemma}
 Given a polygonal tangle with $N$ edges and $c$ crossings, there exists a real
 algebraic parametrization $k$ of a knot and two points $A$  and $B$ in
 $\mathbb{RP}^1$ such that the image under $k$ of one of the two components of
 $\mathbb{RP}^1 \setminus \{A, B\}$ is a section of the tubular neighbourhood of
 the given polygonal tangle and the images of $A$ and $B$ are the end-points of
 the given tangle. 
\end{lemma}

\begin{proof}
  The result is trivially true for a tangle with only one edge: 
  the line containing the edge is a rational knot and let $A$ and $B$ be the
  pre-image of the end-points of the segment.

 We will proceed by induction on the number of edges. Let the given tangle be the union of line segments, $\cup_1^N L_i$. Let $K_n := \cup_1^n L_i$ with the end
 points connected by a line segment.

Assume that the lemma is true for for all tangles with $n$ edges. Then, given a
tangle which is a union of line segments $\cup_1^{n+1} L_i$, there exists a real algebraic knot
$k_1$ and points $A_1$ and $B_1$ of that parameter, so that the image of the
interval $[A_1, B_1]$ lies in the tubular neighbourhood of the tangle $\cup_1^n
L_i$. Reparametrize $k_1$ so that it $A_1$ is the image of $-\infty$ and $B_1$
is the image of 0. Choose a parametrization of the line containing the line
segment $L_{n+1}$ at the point $B_1$ so that the image of 0 is $B_1$  (also the
non-free end-point of $L_1$) and the image of $\infty$ is the free end point of
$L_{n+1}$ which we will call $C$. Then by the theorem on gluing, we know that
there is a knot $k$ in the tubular neighbourhood of $k_1 \cup L$. By the
theorem, we also know that the points $A$ and $C$ are unchanged under the gluing
since they were the images of $-\infty$ and $\infty$. By the gluing theorem, the
image from $A$ to $C$ lies in the tubular neighbourhood of the polygonal tangle.
\end{proof}

\begin{lemma}
Given a polygonal knot with $n \geq 3$ edges and $c$ crossings, there exists a real
algebraic knot which is projectively isotopic to the given knot after some
crossing changes. 
\end{lemma}

\begin{theorem}
    Given a polygonal knot with $n$ edges and $k$ crossings, there exists a real
    algebraic knot of degree~$d \leq {n \choose 2} - k$ in its isotopy class.
  \end{theorem}

  \begin{proof}
  First, we prove that there exists such a knot (see
  Figure~\ref{fig:polygonaltrefoil}). Choose a vertex of the polygonal knot and
  replace it by a very small edge $L$ contained in a small enough ball around
  the vertex so that the resulting knot now has one more edge. On removing the
  edge, one obtains a tangle and therefore by the previous lemma, there exists a
  real algebraic knot and points $A$ and $B$ of the parametrization so that the
  image under  the parametrization of the interval $[A, B]$ is isotopic to the
  tangle. Note that $A$ and $B$ lie in the ball. If we connect the end points of
  the image, under this parametrization, of $[A,B]$, inside the sphere, then
  the resulting topological knot is isotopic to the given polygonal knot. Let us
  call this topological knot $k'$. On the other hand, the image of complement of
  $[A,B]$ can also be joined at the end points to obtain a topological knot
  $k''$ which may be linked to $k'$. By crossing changes, which can be achieved
  by gluing circles to the original real algebraic knot that was constructed, $k'$ and $k''$ can be unlinked. Denote the resulting knots
  by $k_1'$ and $k_1''$. Observe that $k_1'$ is still isotopic to the given
  polygonal knot.  Furthermore, by crossing
  changes, one can ensure that $k_1''$ is the unknot. Therefore, our real
  algebraic knot is isotopic to the connected sum of the given polygonal knot
  and the unknot and is therefore isotopic to the given knot.

Define a sub-polygonal knot of a polygonal knot to be a union of a subset of the segments of the polygonal knot. We will show, by induction, that we can construct a sub-polygonal knot.

  Now we compute the degree: For $n$ edges, $n$ lines are needed. That increases
  the degree to $n$ however some over crossings may need to be changed to under
  crossings and vice-versa in order to unlink and unknot the cylinder knot.
  Observe that each vertex corresponds to an intersection of two lines, which
  accounts for $n$ out of the $n \choose 2$ pairs. Also, $k$ pairs of
  intersections in the projection account for the over and under crossings of the
  original knot.
  \end{proof}

%\section{An alternative approach}
\begin{remark}
 If we weaken the definition of an real algebraic knot, we can reduce the degree. Let us consider $S^1 := [0,1]/0\sim 1$. Viewing $S^1$ in this way, a knot parametrization is equivalent to a map $k : [0, 1] \to \mathbb{RP}^3$ such that $k(0) = k(1)$. Let $q : [0, 1] \to S^1 = [0, 1] / 0 \sim 1$ denote the quotient map from the closed interval to the the circle. If we define  parametrization $k : S^1 \to \mathbb{RP}^3$ to be real algebraic if $k \circ q$ can be represented as a the restriction of a rational map from the projective line to $\mathbb{RP}^3$, then given a polygonal knot with $n$ edges, there exists a real algebraic knot of degree $n+1$ which is differentiable at one point and smooth at all other points. This follows from the previous proof by using part 3 of theorem~\ref{maingluingtheorem}. Such  a knot will not be smooth at one point by choosing $\epsilon_i$ appropriately, one can make the derivatives match and ensure that it is differentiable at that point. 
\end{remark}
%%\begin{lemma}
%%Given a polygonal alternating knot with $n$ edges, such that each edge contains
%%at most one crossing, there exists a real algebraic projective knot of degree~$n$.
%%\end{lemma}
%%
%%\begin{proof}
%%Consider a polygonal knot $K$ such that $K=\cup_1^n L_i$, where each $L_i$ is a
%%line segment in $\mathbb{RP}^3$ and each $L_i$ is adjacent to $L_{i+1}$. Let
%%$K_j=\cup_1^j L_i$. We will prove that if $K_j$ can be constructed by gluing
%%$j$ lines then $K_{j+1}$ can be constructed by gluing $j+1$ lines. 
%%
%%
%%Fix a plane $P$, say the plane at infinity, that is disjoint from the polygonal knot. 
%%
%%Assume that $K_j$ has been constructed by gluing lines.  Now glue another line
%%onto the end point so that its slope is high enough that, along with the
%%previous line segment, it does not interlink with any other line segment. 
%%\end{proof}
%%
%%\begin{lemma}
%%Given a polygonal alternating knot with $n$ edges, such that each edge contains
%%at most one crossing, there exists a real algebraic algebraic knot of degree~$2n - k$, where $k$ is the minimal intersection number of a plane with the knot.
%%\end{lemma}
%%
\subsubsection{Using thin ellipses}
\begin{figure}[h]
  \begin{center}
 \includegraphics[width=10cm]{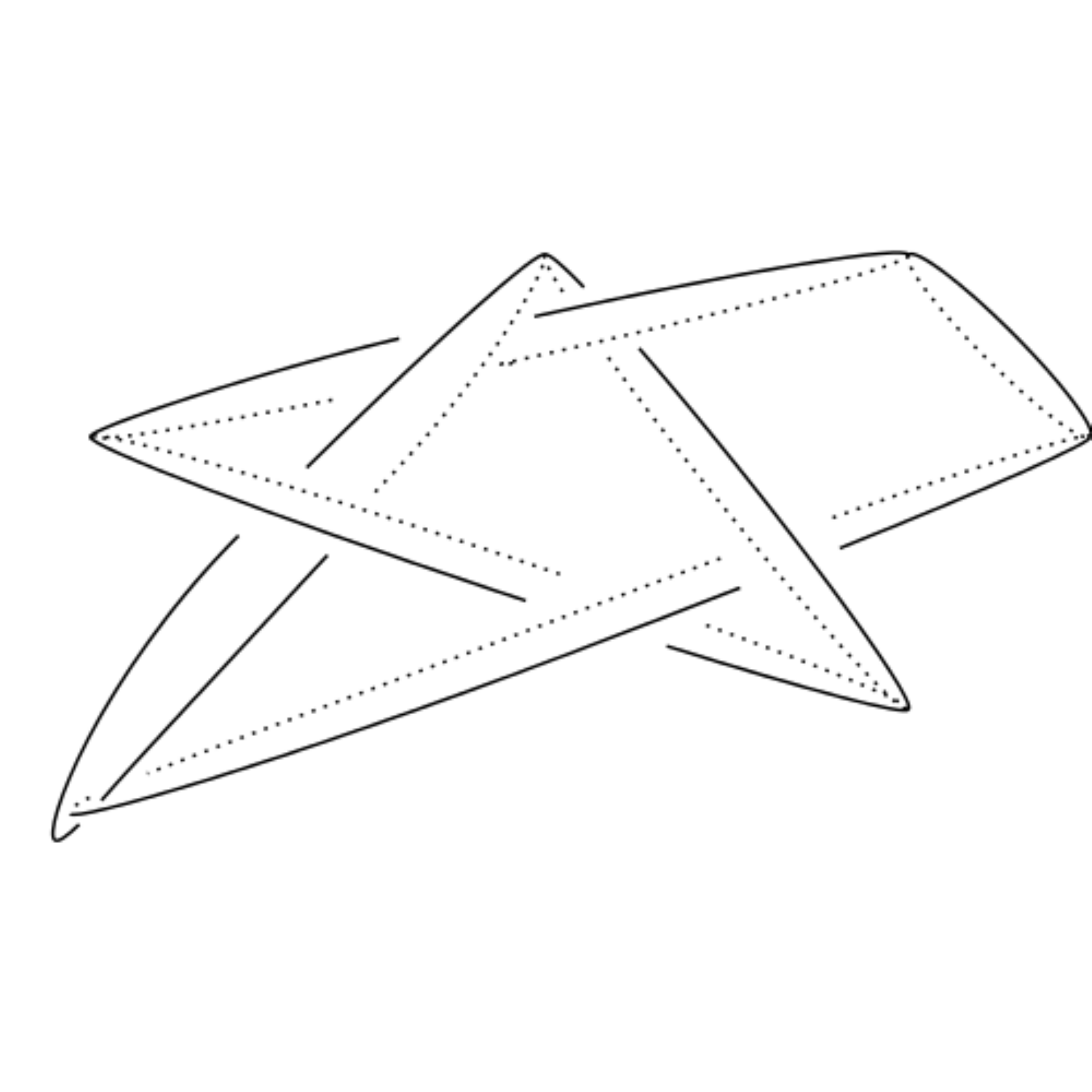} 
  \end{center}
  \label{thinellipse}
  \caption{The double of the knot constructed by gluing thin ellipses. The dotted part is unwanted and can be unlinked by crossing changes
    that can be done by gluing circles}
\end{figure}  

\begin{theorem}
 Given a polygonal knot with $n$ edges and $c$ crossings, there exists a real
 rational knot of degree $2(n+c)$ which is isotopic to it.
\end{theorem}

Recall the definition of a untwisted double knot~\cite[Example~4.D.4]{rolfsen1976knots}
companion is the unknot and satellite is the knot shown in the figure.

We will need the following lemma whose proof is easy:

\begin{lemma}
  An untwisted double knot with crossing number $c$ can be transformed into the
  original knot by a maximum of $c$ crossing changes.
\end{lemma}
\qed

\begin{lemma}
Consider a polygonal knot with $n$  edges and $c$ crossings and ends $A$ and
$B$. There exists a real rational knot of degree $2(n+c)$
isotopic to it this knot.
\end{lemma}

\begin{proof}
  Consider the double of the knot. Observe that it can be constructed by gluing
  using $n$ ellipses as shown in the figure, therefore the resulting real
  algebraic knot has degree $2n$. By the previous lemma, we merely need to
  switch a maximum of $c$ crossings, which can be done by gluing $c$ circles
  (lemma~\ref{switchingcrossings}),
  thereby increasing the degree by a maximum of $2c$. The resulting degree is $2(n+c)$.
\end{proof}

\def\cprime{$'$}

\end{document}